\documentclass[11pt]{amsart}

\usepackage{mathrsfs} 
\usepackage{mathtools}
\usepackage{a4wide} 
\usepackage{graphicx} 
\usepackage{amssymb} 
\usepackage{epstopdf}
\usepackage{enumerate} 
\usepackage{titletoc} 
\usepackage[colorlinks=true, urlcolor=blue, linkcolor=blue, citecolor=black]{hyperref} %
\usepackage[nameinlink]{cleveref} 
\usepackage{esint} 
\usepackage{verbatim}
\usepackage{enumitem}
\usepackage{kotex}

\numberwithin{equation}{section}
\DeclareMathOperator*{\osc}{osc}

\theoremstyle{plain}
\newtheorem{theorem}{Theorem}[section]
\newtheorem{lemma}[theorem]{Lemma}

\theoremstyle{definition}
\newtheorem{definition}[theorem]{Definition}
\newtheorem{remark}[theorem]{Remark}

\makeatletter
\@namedef{subjclassname@2020}{\textup{2020} Mathematics Subject Classification}
\makeatother

\title[Fully nonlinear equations of porous medium-type]{Regularity theory for fully nonlinear equations of porous medium-type}

\author{Hyungsung Yun}
\address{School of Mathematics, Korea Institute for Advanced Study, Seoul 02455, Republic of Korea}
\email{hyungsung@kias.re.kr}


\thanks{Hyungsung Yun has been supported by the KIAS Individual Grant (No. MG097801) at Korea Institute for Advanced Study.}

\begin{document}

\begin{abstract}
In this paper, we establish the regularity results for nonnegative viscosity solutions to fully nonlinear equations of porous medium-type in bounded domains with the zero Dirichlet boundary condition, to be precise, we prove the global $C^{2,\alpha}$-estimates of viscosity solutions. In many PDE problems, the $C^{2,\alpha}$-estimates have been obtained through Schauder-type estimates. However, the Schauder-type estimates are not applicable to the porous medium-type equations. We provide techniques for handling porous medium-type equations so that the global $C^{2,\alpha}$-estimates can be established.

\end{abstract}

\maketitle

%
%

\section{Introduction}
In this paper, we consider the Cauchy--Dirichlet problem for fully nonlinear degenerate parabolic equations of porous medium-type
\begin{equation} \label{prob:main}
	\left\{\begin{aligned}
		u_t &= u^\gamma F(D^2 u,x,t) && \text{in }\Omega\times(-T,0] \\
		u &= u_0&& \text{in } \Omega\times\{ t=-T \} \\
		u &= 0 && \text{on } \partial\Omega\times[-T,0),
	\end{aligned}\right.
\end{equation}
where $0<\gamma<1$, $u_0>0$ in $\Omega$, and $\Omega$ is a bounded domain. Here the fully nonlinear operator $F$ is uniformly parabolic with certain structural conditions (the hypotheses on $F$ will be precisely stated in \Cref{sec:pre}). 

We present some previous literature that provides motivation for the problem \eqref{prob:main}. Using hodograph transform and change of variables that flattens the boundary, the equation in \eqref{prob:main} is converted to a degenerate equation of the form
\begin{equation} \label{intro:ly23}
	v_t = x_n^\gamma \widetilde{F}(D^2v, Dv, x,t) \quad \text{in } Q_1^+ \coloneqq B_1^+ \times (-1,0].
\end{equation}
Since the $Dv$ dependence of $\widetilde F$ can be treated as an $(x,t)$-variable thanks to the interior $C^{1,\alpha}$-regularity of $v$, the global $C^{1,\alpha}$ and $C^{2,\alpha}$-regularity of viscosity solutions to the Cauchy--Dirichlet problem for \eqref{intro:ly23} can be obtained from the results of Lee and Yun \cite{LY23}. Let $u$ be a solution to \eqref{prob:main} and let $v$ be a solution to \eqref{intro:ly23}. Then, the higher regularity of $u$ can be described by the $C_s^{k,\alpha}$-regularity of $v$ for the singular metric
 \begin{equation*}
	s[(x,t),(y,\tau)] \coloneqq \frac{|x-y|}{x_n^{\gamma/2} + y_n^{\gamma/2} + |x'-y'|^{\gamma/2}} + \sqrt{|t-\tau|},
\end{equation*} 
where $C_s^{k,\alpha}(\overline{Q_1^+})$ is the $s$-H\"older space with respect to metric $s$; we refer to \cite{DH98, KL09} and see also \cite{KLY23, KLY24}. Although we can obtain the $C_s^{2,\alpha}$-regularity of $u$ in this way, the $C_s^{2,\alpha}$-regularity still does not deduce the regularity of $u$ in the classical sense. Nevertheless,  results in \cite{LY23} naturally lead to the following question:\\
	
\textbf{Question.} Are solutions to the problem \eqref{prob:main} classical up to the boundary? \\

There is an answer to this question when $F(D^2u,x,t)=\Delta u$. Consider the Cauchy--Dirichlet problem for the porous medium equation of nondivergence form
\begin{equation} \label{prob:FME} 
	\left\{\begin{aligned}
		u_t &= u^\gamma \Delta u && \text{in } \Omega \times (0,\infty) \\
		u &= u_0 && \text{in } \Omega \times \{t=0\}  \\
		u &= 0 && \text{on } \partial \Omega \times [0,\infty).
	\end{aligned}\right.
\end{equation}
For the problem \eqref{prob:FME}, the above question was proved by Jin, Ros-Oton, and Xiong \cite{JROX24}, to be precise, they show that the optimal regularity of weak solutions $u$ to the problem \eqref{prob:FME} with a smooth bounded domain $\Omega$ is 
\begin{equation*}
	\left\{\begin{aligned}
		u(\cdot, t) &\in C^{2,1-\gamma}(\overline\Omega) && \text{for all } t > T^*\\
		u(x, \cdot) &\in C^{\infty}(T^*,\infty) && \text{uniformly in } x \in \overline{\Omega},
	\end{aligned}\right.
\end{equation*}
where the time $T^*>0$ satisfying
\begin{equation} \label{u_sp_lip}
	u(\cdot,t) \in \textnormal{Lip} (\overline{\Omega}) \quad \text{for all } t > T^*.
\end{equation}
Their results \cite{JROX24} lead us to expect that the above question would also be positive, but there are technical differences between the two problems. Since the equation in \eqref{prob:FME} can be transformed into a divergence form, variational techniques are applicable to \eqref{prob:FME}, to be precise, by setting $u=(1-\gamma)^{-1/\gamma}v^{\frac{1}{1-\gamma}}$, the problem \eqref{prob:FME} can be transformed into the porous medium equation of divergence form
\begin{equation*}
	\left\{\begin{aligned}
		v_t &= \Delta v^m&& \text{in } \Omega \times (0,\infty) \\
		v &= v_0 && \text{in } \Omega \times \{t=0\}  \\
		v &= 0 && \text{on } \partial \Omega \times [0,\infty),
	\end{aligned}\right.
\end{equation*}
where $m \coloneqq 1/(1 -\gamma)$ and $u_0 \coloneqq (1-\gamma)^{-1/\gamma}v_0^{\frac{1}{1-\gamma}}$. However, their techniques for weak solutions are not suitable for the problem \eqref{prob:main} of the nondivergence form, thus to answer the question, we need to study the regularity of viscosity solutions to \eqref{prob:main}. 

The goal of this paper is to answer the above question.
\subsection{Main results}
Aronson and Peletier \cite{AP81} proved that there exists a time $T^*$ such that 
\begin{equation} \label{result:ap81}
	u(\cdot,t) \asymp \textnormal{dist}(\cdot,\partial \Omega) \quad \text{for all } t > T^*.
\end{equation}
To answer the question, we assume appropriate condition corresponding to \eqref{u_sp_lip} and  \eqref{result:ap81} in the initial data $u_0$. That is, the time $T^*$ can be understood as the initial time in our scenario.

Here is our main theorem.
\begin{theorem}\label{thm:solv}
	Suppose that $F$ satisfies \textnormal{\ref{F1}}-\textnormal{\ref{F3}} and $u_0 \in C^{2,\alpha}(\overline{\Omega})$ satisfies
\begin{equation} \label{con_u0}
	c_1 \,\textnormal{dist}(x,\partial \Omega) \leq u_0(x) \leq c_2 \,\textnormal{dist}(x,\partial \Omega) \quad \text{for all } x \in \overline{\Omega}
\end{equation}
for some constants $c_1>0$, $c_2>0$, and $\alpha\in(0,1)$. Let $\Omega \subset \mathbb{R}^n$ be any bounded $C^{2,\alpha}$-domain and let $u\in C(\overline{\Omega_T})$ be the nonnegative viscosity solution of \eqref{prob:main}. Then $u \in C^{2, \widetilde\alpha}(\overline{\Omega_T})$ for some $\widetilde\alpha \in (0,\alpha)$ with a uniform estimate
\begin{equation*}
	\|u\|_{C^{2,\widetilde\alpha}(\overline{\Omega_T})} \leq C,
\end{equation*}
where $C>0$ is a constant depending only on  $n$, $\lambda$, $\Lambda$, $\gamma$, $\alpha$, $c_1$, $c_2$, $\|u\|_{L^{\infty}(\Omega_T)}$, $\|u_0\|_{C^{2,\alpha}(\overline{\Omega_T})}$, $\|\beta^1\|_{C^{\alpha}(\overline{\Omega_T})}$, $\|\beta^2\|_{C^{\alpha}(\overline{\Omega_T})}$, and $ \|\partial \Omega\|_{C^{2,\alpha}}$.
\end{theorem}

To prove \Cref{thm:solv}, we first establish the boundary $C^{1,\alpha}$ and $C^{2,\alpha}$-regularity. The constants $\overline{\alpha} \in (0,1)$ used in describing theorems throughout this paper is a universal constant derived from \Cref{result:ly24}.
\begin{theorem}[Boundary $C^{1, \alpha}$-estimates] \label{thm:bdry_c1a}
	Let $\alpha \in(0, \overline{\alpha})$ and $(x_0,t_0) \in \partial \Omega \times (-T,0]$. Suppose  that $\partial \Omega \in C^{1,\alpha}(x_0)$, and  $u_0 \in C(\overline{\Omega})$ satisfies \eqref{con_u0}. Let $u \in C(\overline{\Omega_T})$ be a nonnegative function that satisfies 
\begin{equation} \label{eq:s-class}
	\left\{\begin{aligned}
		u_t &\ge u^{\gamma} \mathcal{M}^{-}_{\lambda, \Lambda}(D^2u) && \text{in the viscosity sense in } \Omega_T \\	
		u_t &\leq u^{\gamma} \mathcal{M}^{+}_{\lambda, \Lambda}(D^2u) && \text{in the viscosity sense in } \Omega_T \\
		u&=u_0 && \text{on } \partial_b \Omega_T \coloneqq \Omega \times \{ t= -T\} \\
		u&=0 && \text{on } \partial_s \Omega_T  \coloneqq  \partial \Omega \times (-T,0).
	\end{aligned}\right.
\end{equation}
Then $u \in C^{1, \alpha}(x_0,t_0)$, that is, there exists a linear function $L$ such that
	\begin{align*}
		|u(x,t)-L(x)| \leq C_\circ (|x-x_0|+\sqrt{|t-t_0|})^{1+\alpha} \quad \text{for all } (x,t) \in \Omega_T \cap Q_1(x_0,t_0) 
	\end{align*}
	and 
	\begin{equation*}
		|Du(x_0,t_0)| \leq C_\circ,
	\end{equation*}
	where $C_\circ>0$ is a constant depending only on $n$, $\lambda$, $\Lambda$, $\gamma$, $\alpha$,  $ \|u\|_{L^{\infty}(\Omega_T)}$, and $\|\partial \Omega\|_{C^{1,\alpha}(x_0)}$.
\end{theorem}
\begin{theorem}[Boundary $C^{2, \alpha}$-estimates] \label{thm:bdry_c2a}
	Let $C_\circ >0$ be as in \Cref{thm:bdry_c1a}, $0< \alpha <\gamma \min\{\overline{\alpha},1-\gamma\}$, and $(x_0,t_0) \in \partial \Omega \times (-T,0]$. Suppose that  $\partial \Omega \in C^{2,\alpha}(x_0)$, $u_0 \in C(\overline{\Omega})$ satisfies \eqref{con_u0}, and $F$ satisfies \textnormal{\ref{F1}}-\textnormal{\ref{F3}}. Let $u \in C(\overline{\Omega_T})$ be a nonnegative viscosity solution of \eqref{prob:main}. Then $u \in C^{2, \alpha}(x_0,t_0)$, that is, there exists a polynomial $P(x)$ with $\deg P \leq 2$ and $F(D^2 P,x_0,t_0) =0$ such that
	\begin{align*}
		|u(x,t)-P(x)|\leq CN(|x-x_0|+\sqrt{|t-t_0|} )^{2+\alpha} \quad \text{for all $(x,t) \in \Omega_T \cap Q_1(x_0,t_0)$,}
	\end{align*}
 and 
	\begin{equation*}
		u_t(x_0,t_0)=0, \quad |Du(x_0,t_0)| + \|D^2u(x_0,t_0)\| \leq CN,
	\end{equation*}
	where the constant $N$ is given by
	\begin{equation*}
		N\coloneqq (\|u\|_{L^{\infty}(\Omega_T)} + [\beta^2]_{C^{\alpha}(x_0,t_0)} +  C_\circ  \|\partial \Omega\|_{C^{2,\alpha}(x_ 0)} + |F(O_n,x_0,t_0)|)
	\end{equation*}
	and $C>0$ is a constant depending only on  $n$, $\lambda$, $\Lambda$, $\gamma$, $\alpha$, and $[\beta^1]_{C^{\alpha}(x_0,t_0)}$.
\end{theorem}
\subsection{Fully nonlinear degenerate parabolic equation}
Let us summarize the contributions of the preceding literature which deals with fully nonlinear degenerate parabolic equations of the form
\begin{equation} \label{eq:deg_sig}
	u_t = \sigma(Du,u,x,t) F(D^2 u ,x, t) + f \quad \text{in } \Omega_T,
\end{equation}
where $\sigma: \mathbb{R}^n \times \overline{\mathbb{R}_+} \times \overline{\Omega_T} \to \mathbb{R}$ is nonnegative function. Note that \eqref{eq:deg_sig} becomes a degenerate parabolic equation when $\sigma(Du,u,x,t)$ vanishes, and those points $(x,t)$ can be interior or boundary points.

When the degeneracy of \eqref{eq:deg_sig} is given as $\sigma(Du,u,x,t)=x_n^\gamma$ and part of the boundary of $\Omega \subset \mathbb{R}^n_+$ is contained in $\{x_n=0\}$, that is, \eqref{eq:deg_sig} is expressed as
\begin{equation} \label{eq:ly23}
	u_t = x_n^\gamma F(D^2 u,x,t) + f \quad \text{in } \Omega_T,
\end{equation}
the boundary $C^{1,\alpha}$ and $C^{2,\alpha}$-regularity of viscosity solutions to \eqref{eq:ly23} was developed in \cite{LY23}. Since the degeneracy of \eqref{eq:ly23} only occurs at the flat boundary $\{x_n=0\}$, it becomes a uniformly parabolic equation in $K_T$ for any $K \subset \joinrel \subset \Omega$, so it is important to study the regularity of viscosity solutions at the boundary points. The results in \cite{LY23} are closely related to \eqref{prob:main} and will actually be used to obtain regularity results in this article. In particular, the regularity result for \eqref{eq:ly23}  can be used to prove the regularity of solutions to the porous medium equation in the case of $F(D^2u ,x,t) = \Delta u$; we refer to \cite{DH98,KL09}.

The next important case is $\sigma(Du, u ,x, t)= |Du|^\gamma$, and the nonlinear operator $F$ depends only on matrix variable, that is, \eqref{eq:deg_sig} is expressed as
\begin{equation}  \label{eq:lly24}
	u_t = |Du|^\gamma F(D^2 u) + f \quad \text{in } \Omega_T.
\end{equation}
Unlike \eqref{eq:ly23}, which was covered in \cite{LY23}, \eqref{eq:lly24} can be degenerate at interior points and the set of degenerate points is not specified, which is one of the difficult factors to prove regularity results. The interior $C^{1,\alpha}$-regularity of viscosity solutions to \eqref{eq:lly24} under the assumption that $F$ is concave or convex with $F \in C^{1,\kappa}(\mathcal{S}^n)$ was developed in \cite{LLY24}. Even if we only assume the uniform ellipticity condition on $F$, the $C^{1,\alpha}$-regularity of solutions can be inferred, but it is still an open problem, and the $C^{1,\alpha}$-regularity of viscosity solutions to 
\begin{equation*}
	u_t = (1+|Du|^2)^{\gamma/2} F(D^2 u) + f \quad \text{in } \Omega_T
\end{equation*} 
has not been developed, even for non-degeneracy $\sigma(Du, u ,x, t)= (1+|Du|^2)^{\gamma/2}$.

On the other hand, the elliptic analogue of \eqref{eq:lly24} has been relatively widely studied in the last decade. To be precise, the interior $C^{1,\alpha}$-regularity of viscosity solutions to 
\begin{equation*}
	|Du|^\gamma F(D^2 u) = f \quad \text{in } B_1
\end{equation*}
was developed in \cite{IS13}.

The case of $\sigma(Du, u, x, t)=u$ is directly related to the porous medium equation. It is quite a delicate issue to discuss the well-posedness of viscosity solutions to these types of equations since it is difficult to obtain the comparison principle, but there are well-known results when $F$ is a smooth linear operator. To be precise, the well-posedness of viscosity solutions to
\begin{equation*}
	u_t = u \Delta u + |Du|^2 \quad \quad \text{in } Q_1
\end{equation*}
was developed in \cite{BV05,GSV99}. When $F$ is a measurable linear operator, we can expect H\"older regularity of viscosity solutions, as in the case of uniformly parabolic equations. For symmetric matrices $A=(a^{ij}(x))$, $B=(b^{ij}(x))$ that satisfy the uniformly ellipticity condition, the H\"older regularity of viscosity solutions to 
\begin{equation} \label{eq:cls23}
	u_t = u a^{ij} D_{ij} u + b^{ij} D_i u D_j u \quad \text{in } Q_1
\end{equation}
was developed in \cite{CLS23}. They developed the modified Krylov-Safonov theory to be applicable to degenerate equations \eqref{eq:cls23} and derived their results. 

In this article, we focus on the case of $\sigma(Du,u,x,t)=u^\gamma$, and although studies on the model equation $u_t =u^\gamma \Delta u$ can be found in many literatures, studies on the fully nonlinear equations $u_t =u^\gamma F(D^2 u)$ are still in initial stages.
\subsection{Strategy of the proof}
Given initial data $u_0$ that is positive in $\Omega$, the viscosity solution $u$ of \eqref{prob:main} is also positive in $\Omega_T$, and hence the equation in \eqref{prob:main} degenerates near $\partial_s \Omega_T$, so the interior regularity for \eqref{prob:main}  follows the results for uniformly parabolic equations. For this reason, \Cref{thm:solv} is proven by obtaining the boundary $C^{2,\alpha}$-regularity and extending it to the global regularity.

In many PDE problems, the $C^{1,\alpha}$ and $C^{2,\alpha}$-regularities have been proven through Schauder-type estimates. However, for the equation in \eqref{prob:main}, these Schauder-type estimates are not applicable due to
the structure of the product of $u^\gamma$ and $F$. In this paper, we provide a technique for handling this product structure so that Schauder-type estimates can be applied.

It is an important observation that the boundary $C^{1,\alpha}$-regularity can be obtained for functions in solution classes, and it can be shown that the function $u$ that satisfies problem \eqref{eq:s-class} belongs to a wider solution class from the boundary behavior of $u$. By the compactness argument, the boundary $C^{1,\alpha}$-regularity of functions in the wider solution class can be obtained from the regularity result of the solution class established in \cite{LY23}. On the other hand, the boundary regularity of functions in a solution class is at most $C^{1,\alpha}$, so the boundary $C^{2,\alpha}$-regularity cannot be obtained with these techniques. To solve this issue, we need additional information about the boundary behavior of $u$ to enable Schauder-type estimates. We converted the structure of the product $u^\gamma$ and $F$ into H\"older continuity of nonlinear operator with degeneracy $[\textnormal{dist}(\cdot,\partial \Omega)]^\gamma$ using the following property:
\begin{equation*}
	\frac{u}{\textnormal{dist}(\cdot,\partial \Omega)} \in C^{\alpha_0}( \overline{\Omega_T \cap Q_{\rho}}) 
	\qquad \text{and} \qquad 
	C_1 \leq \frac{u}{\textnormal{dist}(\cdot,\partial \Omega)} \leq C_2 \quad \text{on }  \in \overline{\Omega_T},
\end{equation*}
for some constants $C_1>0$, $C_2>0$, and $\alpha_0 \in (0,1)$. This property allows us to apply the Schauder-type estimates and obtain global regularity.

\subsection{Organization of the paper}
The paper is organized as follows. In \Cref{sec:pre}, we summarize several notations, definitions, and known results that will be used throughout the paper. 
In \Cref{sec:comp}, we discuss the comparion pinciple with concave solutions, which is used to obtain the boundary Lipschitz estimates.
\Cref{sec:c1a} is devoted to the proof of \Cref{thm:bdry_c1a}  (boundary $C^{1,\alpha}$-estimates). 
In \Cref{sec:dirichlet}, we prove \Cref{thm:bdry_c2a} (boundary $C^{2,\alpha}$-estimates), and then prove our main theorem, \Cref{thm:solv}. 
%
%
\section{Preliminaries} \label{sec:pre}
In this section, we summarize some basic notations and gather definitions and known regularity results used throughout the paper. 

For a point $(x_0,t_0) \in \mathbb{R}^{n+1}$ and $r>0$, we denote the cylinder and upper cylinder as
\begin{equation*}
	Q_r(x_0,t_0) \coloneqq B_r(x_0) \times (t_0 -r^2 , t_0] \quad \text{and} \quad
	Q_r^+(x_0,t_0) \coloneqq Q_r(x_0,t_0) \cap \{ x_n>0\}. 
\end{equation*}
where $B_r(x_0)$ is an open ball with center $x_0$ and radius $r$. For convenience, we denote 
\begin{equation*}
	Q_r = Q_r(0,0) \quad \text{and} \quad Q_r^+ = Q_r^+(0,0). 
\end{equation*}

For general domains, we denote $\Omega_T \coloneqq \Omega\times(-T,0]$ for a bounded domain $\Omega \subset \mathbb{R}^n$ and a time $T>0$.  We define the bottom, corner, side, and parabolic boundary as
\begin{align*}
	\partial_b \Omega_T \coloneqq \Omega \times \{ t= -T\}, \quad
	\partial_c \Omega_T \coloneqq \partial \Omega \times \{ t= -T \}, \quad
	\partial_s \Omega_T \coloneqq  \partial \Omega \times (-T,0),
\end{align*}
and 
$$\partial_p \Omega_T \coloneqq\partial_b \Omega_T  \cap \partial_c \Omega_T \cap \partial_s \Omega_T.$$

\subsection{Hölder continuity}
Defines the H\"older continuity of functions used in the description of the theorem.
\begin{definition}
	Let $A \in \mathbb{R}^{n+1}$ be a bounded set and $f : A \to \mathbb{R}$ be a function. We say that $f$ is $C^{k,\alpha}$ at $(x_0,t_0) \in A$ (denoted by $f \in C^{k,\alpha}(x_0,t_0)$), if there exist constant $C>0$, $r>0$, and a polynomial $P_f(x,t)$ with $\deg P_f \leq k$ such that 
	\begin{equation}\label{cka_f}
		|f(x,t)-P_f(x,t)| \leq C(|x-x_0|+\sqrt{|t-t_0|})^{k+\alpha} \quad \text{for all } (x,t) \in A \cap Q_r(x_0,t_0).
	\end{equation}
	We define 
	\begin{align*}
		[f]_{C^{k,\alpha}(x_0,t_0)} &\coloneqq \inf \{C>0 \mid \eqref{cka_f} \text{ holds with } P_f(x,t) \text{ and } C \}, \\
		\|f\|_{C^{k,\alpha}(x_0,t_0)} &\coloneqq [f]_{C^{k,\alpha} (x_0,t_0)} + \sum_{i=0}^k |D^i P_f(x_0,t_0)|.
	\end{align*}
\end{definition}
We next describe the smoothness of $\partial \Omega$ in the pointwise sense. We note that the following deﬁnition, which was introduced in \cite{LZ20,LZ22}, does not depend on the graph representation of $\partial \Omega$.
\begin{definition}
	Let $\Omega \subset \mathbb{R}^{n}$ be a bounded domain. We say that $\partial \Omega$ is $C^{k,\alpha}$ at $x_0 \in \partial\Omega$ (denoted by $\partial \Omega \in C^{k,\alpha}(x_0)$), if there exist constant $C>0$, $r>0$, a new coordinate system $\{x_1,\cdots,x_n\}$, and a polynomial $P_{\partial \Omega}(x')$ with $\deg P_{\partial \Omega} \leq k$, $P_{\partial \Omega}(0')=0$, and $DP_{\partial \Omega}(0')=0$ such that $x_0=0$ in this coordinate system,
	\begin{equation}\label{cka_dom}
		\begin{aligned}
			&B_\rho \cap \{ (x',x_n) \mid x_n > P_{\partial \Omega}(x') + C|x'|^{k+\alpha} \} \subset B_r \cap \Omega, \\
			&B_\rho \cap \{ (x',x_n) \mid x_n < P_{\partial \Omega}(x') - C|x'|^{k+\alpha} \} \subset B_r \cap \Omega^c.
		\end{aligned}
	\end{equation}
	We define 
	\begin{align*}
		[\partial \Omega]_{C^{k,\alpha}(x_0)} &\coloneqq \inf \{C>0 \mid \eqref{cka_dom} \text{ holds with } P_{\partial \Omega}(x',t) \text{ and } C \}, \\
		\|\partial \Omega\|_{C^{k,\alpha}(x_0)} &\coloneqq [\partial \Omega]_{C^{k,\alpha} (x_0)}  + \sum_{i=2}^k |D^i P_{\partial \Omega}(0')|.
	\end{align*}
Moreover, if $\partial \Omega \in C^{k,\alpha}(a)$ for all $a\in \partial \Omega$, we define 
\begin{equation*}
	\|\partial \Omega\|_{C^{k,\alpha}}\coloneqq \sup_{a\in\partial \Omega} \|\partial \Omega\|_{C^{k,\alpha}(a)}.
\end{equation*}
\end{definition}
\subsection{Viscosity solutions}
We now introduce the concept of viscosity solutions of fully nonlinear parabolic equations:
\begin{equation}\label{eq:model}
	u_t = u^\gamma F(D^2u,x,t) ,
\end{equation}
where $0<\gamma <1$ and we assume that the fully nonlinear operator $F:\mathcal{S}^n \times \Omega_T \to \mathbb{R}$ is continuous and satisfies the following condition:
\begin{equation*}
	M \leq N \quad \Longrightarrow \quad F(M,x,t) \leq  F(N,x,t) \quad \text{for all } (x,t) \in \Omega_T. 
\end{equation*}
\begin{definition} [Test functions]
	Let $u$ be a continuous function in $\Omega_T$. The function $\varphi : \Omega_T \to \mathbb{R}$ is called \textit{test function} if it is $C^1$ with respect to $t$ and $C^2$ with respect to $x$.
	\begin{enumerate} [label=(\roman*)]
		\item We say that the test function $\varphi$ touches $u$ from above at $(x,t)$ if there exists an open neighborhood $U$ of $(x,t)$ such that 
		$$0 < u \leq \varphi  \quad \mbox{in } U \qquad  \mbox{and} \qquad u(x,t) = \varphi(x,t). $$
		\item We say that the test function $\varphi$ touches $u$ from below at $(x,t)$ if there exists an open neighborhood $U$ of $(x,t)$ such that 
		$$u \ge \varphi > 0 \quad \mbox{in } U \qquad  \mbox{and} \qquad u(x,t) = \varphi(x,t). $$
	\end{enumerate}
\end{definition}
We now give the definitions of viscosity solutions of \eqref{eq:model}. These definitions are similar to the definitions in \cite{BV05,GSV99} for linear equations and \cite{CLS23} for nonlinear equations.
\begin{definition}[Viscosity solutions] 
	Let $F:\mathcal{S}^n\times \Omega_T \to \mathbb{R}$ be a fully nonlinear parabolic operator. 
	\begin{enumerate} [label=(\roman*)]
		\item Let $u:\Omega_T \to\mathbb{R}$ be a nonnegative upper semicontinuous function. We say $u$ is a \textit{viscosity subsoution} of \eqref{eq:model} in $\Omega_T$ provided the following condition holds: if for any $(x,t) \in \Omega_T$ and any test function $\varphi$ touching $u$ from above at $(x,t)$, then
		\begin{equation*}
			\varphi_t (x,t) \leq \varphi (x,t)^\gamma F ( D^2 \varphi (x,t),x,t ) .
		\end{equation*}
		\item Let $u:\Omega_T \to \mathbb{R}$ be a nonnegative lower semicontinuous function. We say $u$ is a \textit{viscosity supersoution} of \eqref{eq:model} in $\Omega_T$  provided the following condition holds: if for any $(x,t) \in \Omega_T$ with $u(x,y)>0$ and any test function $\varphi$ touching $u$ from below at $(x,t)$, then
		\begin{equation*}
			\varphi_t (x,t) \ge \varphi (x,t)^\gamma F ( D^2 \varphi (x,t), x,t ) .
		\end{equation*}
	\end{enumerate}
\end{definition}
\subsection{Structure hypotheses on the operator $F$}  
We introduce the Pucci's extremal operators, which are important operators in the study of fully nonlinear equations.

\begin{definition} [Pucci's extremal operators] 
Given ellipticity constants $0<\lambda \leq \Lambda$ and $\mathcal{S}^n \coloneqq \{ M : \text{$M$ is an $n\times n$ real symmetric matrix\}}$, we define \textit{Pucci's extremal operators} as follows:
\begin{align*}
	\mathcal{M}_{\lambda, \Lambda}^{+}(M) \coloneqq \sup_{\lambda I \leq A \leq \Lambda I} \text{tr} (AM) \quad \text{and} \quad 
	\mathcal{M}^{-}_{\lambda, \Lambda}(M)  \coloneqq \inf_{\lambda I \leq A \leq \Lambda I} \text{tr} (AM).
\end{align*}
\end{definition}

We assume that the fully nonlinear operator $F:\mathcal{S}^n \times \overline{\Omega_T} \to \mathbb{R}$ satisfies some of the following conditions:
\begin{enumerate} [label=\text{(F\arabic*)}]
	\item \label{F1} $F$ is uniformly parabolic with $F(O_n,\cdot)\equiv0$; that is,  there exist constants $0<\lambda \leq \Lambda$ such that \begin{equation*}
	\mathcal{M}^{-}_{\lambda, \Lambda}(M-N) \leq F(M,x,t) - F(N,x,t) \leq \mathcal{M}^{+}_{\lambda, \Lambda}(M-N)
\end{equation*}
for all $M,N \in \mathcal{S}^n$ and $(x,t) \in\overline{\Omega_T}$.
	\item \label{F2} $F$ is concave or convex.
	\item \label{F3} Let $\alpha \in (0,1)$ and $(x_0,t_0)\in \overline{\Omega_T}$. There exist a constant $r_0>0$ and a nonnegative function $\beta^1, \beta^2 \in C^{\alpha}(\overline{\Omega_T})$ such that $\beta^1(x_0,t_0)=0=\beta^2(x_0,t_0)$ and
\begin{equation*}
	|F(M,x,t)-F(M,x_0,t_0)| \leq \beta^1(x,t) \|M\| + \beta^2(x,t)
\end{equation*}
for all $M\in\mathcal{S}^n$ and $(x,t) \in \overline{\Omega_T \cap Q_{r_0}(x_0,t_0)}$.
\end{enumerate}
\subsection{Known regularity results}
Here we gather some known regularity results that we will need later on. 
\begin{theorem} \cite[Theorem 3.4]{KL13} \label{result:kl13}
Suppose that $F$ is positively homogeneous of degree 1, i.e., 
\begin{equation*}
	F(rM) = r F(M) \quad \text{for all } r>0 \text{ and } M\in \mathcal{S}^n.
\end{equation*}
Then the Dirichlet problem
\begin{equation*} 
	\left\{\begin{aligned}
		F(D^2 \phi) + \gamma^{-1}(1-\gamma) \phi^{1-\gamma}  &= 0 && \text{in } \Omega\\
		\phi &= 0 && \text{on } \partial \Omega
	\end{aligned} \right.
\end{equation*}
has a unique viscosity solution $\phi \in C^{0,1} (\overline{\Omega})\cap C^{1,\alpha}(\Omega)$ that satisfies
\begin{equation*}
	C_1 \,\textnormal{dist}(x,\partial \Omega) \leq \phi(x) \leq C_2 \,\textnormal{dist}(x,\partial \Omega) \quad \text{for all } x \in \overline{\Omega},
\end{equation*}
where $C_1>0$ and $C_2>0$ are constants depending only on $n$, $\lambda$, $\Lambda$, $\gamma$, and $\|\phi\|_{L^\infty(\Omega)}$.
\end{theorem}
\begin{theorem} \cite[Lemma 3.4]{LY23} \label{result:ly24}
	Let $u \in C(\overline{Q_1^+})$ satisfy $u=0$ on $\{x_n=0\}$ and 
\begin{equation*}
	x_n^\gamma \mathcal{M}_{\lambda, \Lambda}^{-}(D^2 u)  \leq u_t \leq x_n^\gamma \mathcal{M}_{\lambda, \Lambda}^{+}(D^2 u)
\end{equation*}
in the viscosity sense in $Q_1^+$. Then there exist $\overline{\alpha}\in(0,1)$ depending only on $n$, $\lambda$, $\Lambda$, and $\gamma$  and a constant $a \in \mathbb{R}$ such that
\begin{equation*}
	|u(x,t)-ax_n| \leq C_*\|u\|_{L^{\infty}(Q_1^+)}(|x| + \sqrt{|t|})^{\overline{\alpha}} x_n \quad \text{for all } (x,t) \in \overline{Q_{1/2}^+}
\end{equation*}
and $|a| \leq C_*$, where $C_*>1$ is a constant depending only on $n$, $\lambda$, $\Lambda$, $\gamma$, and $\overline{\alpha}$.
\end{theorem}
\begin{theorem} \cite[Theorem 1.1]{LY23} \label{result:ly24-3}
	Let $\alpha \in(0, \overline{\alpha})$ with $\alpha \leq 1-\gamma$. Suppose that $f\in C(Q_1^+) \cap L^{\infty}(Q_1^+)$ and $u \in C(\overline{Q_1^+})$ satisfy  $u=0$ on $\{x_n=0\}$ and
	\begin{equation} \label{s*_class}
		x_n^\gamma \mathcal{M}_{\lambda, \Lambda}^{-}(D^2 u) -\|f\|_{L^\infty(Q_1^+)} \leq u_t \leq x_n^\gamma \mathcal{M}_{\lambda, \Lambda}^{+}(D^2 u) + \|f\|_{L^\infty(Q_1^+)} 
	\end{equation}
in the viscosity sense in $Q_1^+$. Then there exists a constant $a \in \mathbb{R}$ such that
	\begin{align*}
		|u(x,t)-ax_n|\leq C_\diamond(\|u\|_{L^{\infty}(Q_1^+)}+\|f\|_{L^{\infty}(Q_1^+)} )(|x|+\sqrt{|t|})^{1+\alpha} \quad \text{for all $(x,t) \in \overline{Q_{1/2}^+}$}
	\end{align*}
	and 
	\begin{equation*}
		|a| \leq C(\|u\|_{L^{\infty}(Q_1^+)}+\|f\|_{L^{\infty}(Q_1^+)}),
	\end{equation*}
	where $C_\diamond>0$ is a constant depending only on $n$, $\lambda$, $\Lambda$, $\gamma$, and $\alpha$.
\end{theorem}
\begin{remark} \label{rem:xn1a-growth}
Let $u$ be a function that satisfies the assumptions of \Cref{result:ly24-3}. By \Cref{result:ly24-3}, we have $u \in C^{1,\alpha}(0,0)$ and 
\begin{equation} \label{inq:xn1+a}
	|u(0',x_n,0)-u_n(0,0)x_n|\leq C_\diamond(\|u\|_{L^{\infty}(Q_1^+)}+\|f\|_{L^{\infty}(Q_1^+)} )x_n^{1+\alpha} \quad \text{for all $x_n  \in [0,1/2]$}.
\end{equation}
Let $v(x,t)=u(x'+y',x_n,t+\tau)$ be a translation for $y' \in \mathbb{R}^{n-1}$. Then $v$ satisfies $v=0$ on $\{x_n=0\}$ and \eqref{s*_class} in $Q_r^+$ for some $r \in (0,1)$ and $\tilde{f}(x,t) \coloneqq f(x'+y',x_n,t+\tau)$ instead of $f$. Thus, the estimate \eqref{inq:xn1+a} holds for $v$, i.e.,  $v\in C^{1,\alpha}(0,0)$ and 
\begin{equation*}
	|u(y',x_n,\tau)-u_n(y',0,\tau)x_n|\leq C_\diamond(\|u\|_{L^{\infty}(Q_1^+)}+\|f\|_{L^{\infty}(Q_1^+)} )x_n^{1+\alpha} \quad \text{for all $x_n  \in [0,r/2]$}.
\end{equation*}

\end{remark}
\begin{theorem} \cite[Lemma 4.1]{LY23} \label{result:ly24-2}
Suppose that $F:\mathcal{S}^n \to \mathbb{R}$ is concave or convex. Let $u \in C(\overline{Q_1^+})$ be a viscosity solution of  
\begin{equation*} 
	\left\{\begin{aligned}
		u_t&=x_n^{\gamma} F(D^2 u) && \text{in } Q_1^+ \\
		u&=0 && \text{on } \{x_n=0\}.
	\end{aligned}\right.
\end{equation*}
Then $u \in C^{2,\alpha}(0,0)$ for any $\alpha \in(0,\overline{\alpha})$ with $\alpha\leq1-\gamma$, i.e., there exists a polynomial $P(x)$ with $\deg P \leq 2$ such that
\begin{equation*}
	|u(x,t)-P(x)| \leq C_\star( \|u\|_{L^{\infty}(Q_1^+)} + |F(O_n)| )  |x|^{1+\alpha} x_n \quad \text{for all } (x,t) \in \overline{Q_{1/2}^+},
\end{equation*}
$\|P\| \leq C_\star$ and $F(D^2 P) =0$, where $C_\star>1$ is a constant depending only on $n$, $\lambda$, $\Lambda$, $\gamma$, and $\alpha$. 
\end{theorem}
%
%
\section{Comparison with concave solutions} \label{sec:comp}
In this section, we discuss a conditional comparison principle for \eqref{eq:model}. In general, the comparison principle for fully nonlinear equations can be obtained under appropriate structure conditions.

We say the fully nonlinear operator $F:\mathcal{S}^n\times \mathbb{R}^n \times \mathbb{R} \times \Omega_T \to \mathbb{R}$ is \textit{proper} provided that for each $(M,p,x,t) \in \mathcal{S}^n \times \mathbb{R}^n \times \Omega_T$,
\begin{equation*} 
	F(M,p,s,x,t) \leq F(M,p,r,x,t) \quad \text{whenever } r \leq s.
\end{equation*}
Theorem 8.2 in \cite{CIL92} provides a comparison principle for many kinds of proper $F$, but does not provide a comparison principle for \eqref{eq:model} because the fully nonlinear operator in \eqref{eq:model} is not proper. However, as a special case, a conditional comparison principle can be obtained if one of the two viscosity solutions is concave in $x$. To demonstrate a conditional comparison principle of \eqref{eq:model}, we introduce semijets and theorems suggested in \cite{CIL92}.
\begin{definition}[Parabolic semijets]\label{semijet}
	 Let $u$ be a function defined in $\Omega_T$ and let $(x, t)\in \Omega_T$.
	\begin{enumerate}[label=(\roman*)]
		\item A \textit{parabolic superjet} $\mathscr{P}^{2, +}u(x, t)$ consists of $(a, p, M) \in \mathbb{R} \times \mathbb{R}^n \times \mathcal{S}^n$ which satisfy
		\begin{align*}
			u(y, s) &\leq u(x,t)+a(s-t)+p \cdot(y-x) \\
			&\qquad +\frac{1}{2} (y-x)^T M(y-x) +o(|s-t|+|z-x|^2) \quad \text{as $(y, s) \to (x, t)$}. 
		\end{align*}
		Similarly, we can define a \textit{parabolic subjet} $\mathscr{P}^{2, -}u(x, t)$. It immediately follows that
		\begin{align*}
			\mathscr{P}^{2, -}u(x, t)=-\mathscr{P}^{2, +}(-u)(x, t).
		\end{align*} 
	
	\item A \textit{limiting superjet} $\overline{\mathscr{P}}^{2, +}u(x, t)$ consists of $(a, p, M) \in \mathbb{R} \times \mathbb{R}^n \times \mathcal{S}^n$ with the following statement holds: there exists a sequence $\{(x_n, t_n, a_n, p_n, M_n)\}_{n=1}^{\infty}$ such that $(a_n, p_n, M_n) \in \mathscr{P}^{2, +}u(x_n, t_n)$ and 
	\begin{equation*}
		\text{$(x_n, t_n, u(x_n, t_n), a_n, p_n, M_n) \to (x, t, u(x, t), a, p, M)$} \quad \text{as } n \to \infty.
	\end{equation*}
	We define a \textit{parabolic subjet} $\overline{\mathscr{P}}^{2, -}u(x, t)$ in a similar way.
	\end{enumerate}	
\end{definition}
\begin{theorem}[Jensen--Ishii's Lemma,  {\cite[Theorem 8.3]{CIL92}}]\label{lem:ishii}
Let $U$ and $V$ be two open sets of $\mathbb{R}^n_+$ and $I$ an open interval of $\mathbb{R}$. Consider also a viscosity subsolution $u$ of \eqref{eq:model} in $U\times I$ and a viscosity supersolution $v$ of \eqref{eq:model} in $V\times I$. Suppose that 
\begin{equation} \label{ishii_fcn}
	w(x,y,t)\coloneqq u(x,t)-v(y,t) - \frac{1}{2\varepsilon} |x-y|^2
\end{equation} 
has a local maximum at $(x_\varepsilon,y_\varepsilon,t_\varepsilon) \in U\times V \times I$. Then there exists $a \in \mathbb{R}$ and $M,N \in \mathcal{S}^n$ such that $(a,p,M) \in \overline{\mathscr{P}}^{2, +}u(x_\varepsilon, t_\varepsilon)$, $(a,p,N) \in \overline{\mathscr{P}}^{2, -}v(y_\varepsilon, t_\varepsilon)$, and 
\begin{equation*}
	-\frac{2}{\varepsilon} 
		\begin{pmatrix}
			I_n & O_n  \\
			O_n & I_n
		\end{pmatrix} \leq 
		\begin{pmatrix}
			M & O_n  \\
			O_n & -N
		\end{pmatrix} \leq  \frac{3}{\varepsilon} 
		\begin{pmatrix}
			I_n & -I_n  \\
			-I_n & I_n
		\end{pmatrix},
\end{equation*}
where $p \coloneqq \varepsilon^{-1}(x_\varepsilon-y_\varepsilon)$, $I_{n}$ and $O_{n}$ are the indentity matrix and zero matrix, respectively.
\end{theorem} 
\begin{lemma} [{\cite[Lemma 3.1]{CIL92}}, {\cite[Lemma 2.3.19]{BEG13}}] \label{lem:mu_ep}
	Let $u$ be a upper semicontinuous in $\overline{\Omega_T}$ and $v$ be a lower semicontinuous in $\overline{\Omega_T}$. Assume that 
	\begin{equation*}
		\mu_{\varepsilon} \coloneqq \sup_{(x,t),(y,s) \in \overline{\Omega_T}} \left(u(x,t)-v(y,s) -\frac{1}{2\varepsilon}|x-y|^2   -\frac{1}{2\varepsilon}|t-s|^2 \right) < \infty
	\end{equation*}
	for small $\varepsilon>0$ and $\{(x_\varepsilon, t_\varepsilon, y_\varepsilon,s_\varepsilon)\}$ satisfies that 
	\begin{equation*}
		\lim_{\varepsilon \to 0} \left( \mu_{\varepsilon} - u(x_\varepsilon,t_\varepsilon)+v(y_\varepsilon,s_\varepsilon) + \frac{1}{2\varepsilon}|x_\varepsilon-y_\varepsilon|^2 + \frac{1}{2\varepsilon}|t_\varepsilon-s_\varepsilon|^2\right)=0.
	\end{equation*}
	Then the following statements hold:
	\begin{enumerate}[label=(\roman*)]
	\item $\displaystyle\lim_{\varepsilon \to 0} {\varepsilon}^{-1} |x_\varepsilon - y_\varepsilon|^2 =0$ and $\displaystyle\lim_{\varepsilon \to 0} {\varepsilon}^{-1} |t_\varepsilon - s_\varepsilon|^2 =0$,
	\item $\displaystyle\lim_{\varepsilon \to 0} \mu_{\varepsilon} = u(\overline{x},\overline{t})-v(\overline{x},\overline{t}) = \sup_{\Omega_T}(u-v)$ whenever $(\overline{x},\overline{t}) \in \overline{\Omega_T}$ is a limit point of $\{(x_\varepsilon, t_\varepsilon)\}$ as $\varepsilon \to 0$.
	\end{enumerate}
\end{lemma} 
The fully nonlinear operator in \eqref{eq:model} satisfies that for each $(M,x,t) \in \mathcal{S}^n \times \Omega_T$ with $M\leq 0$,
\begin{equation*}
	s^\gamma F(M,x,t) \leq r^\gamma F(M,x,t) \quad \text{whenever } r \leq s.
\end{equation*}
and we obtain the following conditional comparison principle.
\begin{lemma}[Comparison with concave supersolution]\label{com_prin}
Suppose that $F$ satisfies \textnormal{\ref{F1}}.
Let $u \in C(\overline{\Omega_T})$ and $v\in C(\overline{\Omega_T}) $ be a viscosity subsolution and viscosity supersolution of \eqref{eq:model} in $\Omega_T$, respectively. If $u \leq v$ on $\partial_p \Omega_T$ and  $F(D^2v,x,t) < 0$ in the viscosity sense in $\Omega_T$, then $u \leq v$ in $\Omega_T$.
\end{lemma} 

\begin{proof}
	We first observe that for a matrix $A\ge 0$ and $\delta>0$, the function $\tilde{v} \coloneqq v + \varphi$ satisfies
\begin{align*}
	\tilde{v}_t 
	&\ge (\tilde{v}-\varphi)^{\gamma} F(D^2 \tilde{v} + \delta A,x,t) + \delta/T^2 \\
	&\ge \tilde{v}^{\gamma} F(D^2 \tilde{v} + \delta A,x,t) + \delta/T^2 \quad \text{in the viscosity sense in } \Omega_T 
\end{align*}
and $\tilde{v} \to \infty$ uniformly on $\overline{\Omega}$ as $t \to 0^-$, where
\begin{equation*}
	\varphi(x,t) \coloneqq  - \frac{\delta}{t} + \delta(\|p\|_{L^{\infty}(\Omega)}  - p(x)) \quad \text{and} \quad p(x) \coloneqq c + p \cdot x +\frac{1}{2} x^T A x.
\end{equation*}
Note also that $u$ satisfies 
\begin{align*}
	u_t &\leq u^{\gamma} F(D^2u,x,t) \\
	& \leq u^{\gamma} F(D^2u+ \delta A,x,t) + \delta/T^2 \quad \text{in the viscosity sense in } \Omega_T .
\end{align*}

	We argue by contradiction. Suppose that $m \coloneqq u(x_0,t_0)-\tilde{v}(x_0,t_0) = \sup_{\Omega_T}(u-\tilde{v}) >0$ for some $(x_0,t_0) \in \Omega_T$. Since $\tilde{v}$ is bounded below, $w$ defined as in \eqref{ishii_fcn} has the maximum $\mu_\varepsilon$  at $(x_\varepsilon,y_\varepsilon,t_\varepsilon)$. Furthermore, $\{(x_\varepsilon,y_\varepsilon, t_\varepsilon)\}$ is a bounded sequence, there exists a subsequence $\{(x_{\varepsilon_j},y_{\varepsilon_j},t_{\varepsilon_j})\}$ such that 
\begin{equation*}
	\lim_{\varepsilon_j \to 0^+} x_{\varepsilon_j} = \overline{x} , \quad 
	\lim_{\varepsilon_j \to 0^+} y_{\varepsilon_j} = \overline{y} ,\quad \text{and} \quad
	\lim_{\varepsilon_j \to 0^+} t_{\varepsilon_j} = \overline{t}.
\end{equation*}
For convenience, subindex $j$ will be omitted from now on. 
If $t_\varepsilon = -T$, by \Cref{lem:mu_ep}, we have $\overline{x} =\overline{y}$ and 
	\begin{align*}
		0<m = \lim_{\varepsilon \to 0^+} \mu_{\varepsilon} &\leq \lim_{\varepsilon \to 0^+}\sup_{x,y \in \Omega} \left(u(x,-T)- \tilde{v}(y,-T) -\frac{1}{2\varepsilon}|x-y|^2  \right) \\
		&=u(\overline{x},-T) -\tilde{v}(\overline{x},-T) <0
	\end{align*}
	which is not possible. Similarly, we can see that $x_\varepsilon, y_\varepsilon \in \Omega$ if $\varepsilon>0$ is sufficiently small. Thus we can apply \Cref{lem:ishii} to the point $(x_\varepsilon,y_\varepsilon,t_\varepsilon)$, there exists $a \in \mathbb{R}$ and $M,N \in \mathcal{S}^n$ such that $(a,p,M) \in \overline{\mathscr{P}}^{2, +}u(x_\varepsilon, t_\varepsilon)$, $(a,p,N) \in \overline{\mathscr{P}}^{2, -} \tilde{v}(y_\varepsilon, t_\varepsilon)$, and 
$M \leq N$, where $p \coloneqq \varepsilon^{-1}(x_\varepsilon-y_\varepsilon)$. These imply that

\begin{equation*}
	a \leq u(x_\varepsilon, t_\varepsilon)^{\gamma} F(M,x_\varepsilon, t_\varepsilon) \quad \text{and} \quad
	a \ge \tilde{v}(y_\varepsilon, t_\varepsilon)^{\gamma} F(N+\delta A,y_\varepsilon, t_\varepsilon) + \delta/T^2.
\end{equation*}
Since $\lim_{\varepsilon \to 0}\varepsilon^{-1} |x_\varepsilon - y_\varepsilon|^2 = 0$ and $M \leq N$, we have
\begin{align*}
	\delta/T^2 & \leq u(x_\varepsilon, t_\varepsilon)^{\gamma} F(M,x_\varepsilon, t_\varepsilon) - \tilde{v}(y_\varepsilon, t_\varepsilon)^{\gamma} F(N+\delta A,y_\varepsilon, t_\varepsilon) .
\end{align*}
Letting $\varepsilon \to 0^+$, we deduce
\begin{equation} \label{after_limit}
	0<\delta/T^2 \leq u(\overline{x} , \overline{t} )^{\gamma} F(M,\overline{x} , \overline{t})  - \tilde{v}(\overline{x} , \overline{t})^{\gamma} F(N+\delta A,\overline{x} , \overline{t}) .
\end{equation}
If $F(M,\overline{x} , \overline{t}) \leq 0$, then taking $A=\delta^{-1} (N-M)$ contradicts \eqref{after_limit}. 
On the other hand, if $F(M,\overline{x} , \overline{t}) > 0$, then taking $A=\delta^{-1} (\mu-1) N$ for $\mu> \max\{1,  \lambda^{-1}\tilde{v}(\overline{x} , \overline{t})^{-\gamma}(1+\Lambda u(\overline{x} , \overline{t})^{\gamma} )\}$, we have
\begin{align*}
	0<\delta/T^2 &\leq u(\overline{x} , \overline{t} )^{\gamma} F(N,\overline{x} , \overline{t})  - \tilde{v}(\overline{x} , \overline{t})^{\gamma} F(\mu N,\overline{x} , \overline{t}) \\
	&\leq (\Lambda u(\overline{x} , \overline{t} )^{\gamma} - \lambda \mu \tilde{v}(\overline{x} , \overline{t})^{\gamma} ) \|N\| \leq - \|N\|
\end{align*} 
which leads to a contradiction. Finally, by letting $\delta \to 0^+$, we obtain the desired conclusion.
\end{proof}
\begin{remark} [Comparison with concave subsolution] \label{com_prin2}
Although we have proven that the right-hand side of \eqref{after_limit} is negative by taking a certain $A$, in fact the same conclusion can be obtained if $\|A\|$ is sufficiently large. 

We now discuss the condition for the comparison with concave subsolutions. Let $u \in C^2(\overline{\Omega_T})$ be a viscosity subsolution of \eqref{eq:model} in $\Omega_T$ such that $F(D^2u,x,t) < 0$ in $\Omega_T$. Then for $A$ with sufficiently large $\|A\|$, we have $F(D^2u+ \delta A,x,t) > 0$ in $\Omega_T$ and hence
\begin{align*}
	u_t &\leq u^{\gamma} F(D^2u,x,t)  \\
	& \leq (u^{\gamma} - \varphi^\gamma) F(D^2u,x,t)  + \delta/T^2 \\
	& \leq |u - \varphi|^\gamma F(D^2u+ \delta A,x,t)  + \delta/T^2 \quad \text{in } \Omega_T. 
\end{align*}
In other words, the concavity assumption of the viscosity supersolution in \Cref{com_prin} can be replaced by the concavity of the $C^2$-subsolution to 
\begin{equation} \label{eq:absol}
	u_t = |u - \varphi|^\gamma F(D^2u+ \delta A,x,t)+ \delta/T^2 \quad \text{in } \Omega_T,
\end{equation}
together with the fact that $v$ is a viscosity supersolution to \eqref{eq:absol}. 
\end{remark}
%
%
\section{Boundary $C^{1,\alpha}$-estimates} \label{sec:c1a}
In this section, we prove the boundary $C^{1,\alpha}$-estimates for functions that satisfy \eqref{eq:s-class}. The inequalities in \eqref{eq:s-class} has degeneracy  $u^\gamma$ near $\partial_s \Omega_T$, and to obtain the boundary $C^{1,\alpha}$-estimates, Lipschitz estimates for $u$ near $\partial_s \Omega_T$ should be studied first. The following lemma shows that Lipschitz estimates for initial data $u_0$ near $\partial \Omega$ are persistently preserved for all time afterward.

Throughout this paper, we assume that $u_0$ is continuous and satisfies \eqref{con_u0}. 

\begin{lemma} \label{lem:u<x}
Let $u \in C(\overline{\Omega_T})$ be a nonnegative function that satisfies \eqref{eq:s-class}. Then 
\begin{equation*}
	C_1 \,\textnormal{dist}(x,\partial \Omega) \leq u(x,t) \leq C_2 \,\textnormal{dist}(x,\partial \Omega) \quad \text{for all } (x,t) \in \overline{\Omega_T},
\end{equation*}
where $C_1>0$ and $C_2>0$ are constants depending only on $n$, $\lambda$, $\Lambda$, $\gamma$, $c_1$, $c_2$, and $T$.
\end{lemma}

\begin{proof}
By \Cref{result:kl13}, there exist solutions $\phi^{\pm}  \in C^{0,1} (\overline{\Omega})\cap C^{2,\alpha}(\Omega)$ of 
\begin{equation*} 
	\left\{\begin{aligned}
		\mathcal{M}^{\pm}_{\lambda,\Lambda}(D^2 \phi) + \gamma^{-1}(1-\gamma) \phi^{1-\gamma}  &= 0 && \text{in } \Omega\\
		\phi &= 0 && \text{on } \partial \Omega,
	\end{aligned} \right.
\end{equation*}
respectively such that 
\begin{equation*}
	C_1 \,\textnormal{dist}(x,\partial \Omega) \leq \phi^{\pm}(x) \leq C_2 \,\textnormal{dist}(x,\partial \Omega) \quad \text{for all } x \in \overline{\Omega}.
\end{equation*}
From \eqref{con_u0}, we can choose $\tau_->0$ and $\tau_+>0$ such that 
\begin{equation} \label{comp_init}
	(1-\gamma)^{-1/\gamma}\tau_-^{-1/\gamma}\phi^-(x) \leq u_0(x) \leq (1-\gamma)^{-1/\gamma} \tau_+^{-1/\gamma}\phi^+(x) \quad \text{for all } x \in \overline{\Omega}.
\end{equation}
Moreover, we can see that $\Phi^{\pm}(x,t) \coloneqq (1-\gamma)^{-1/\gamma}  (t + T+ \tau_{\pm})^{-1/\gamma}\phi^{\pm}(x)$ are solutions of 
\begin{equation*}
	\Phi_t = \Phi^\gamma \mathcal{M}^{\pm}_{\lambda,\Lambda}(D^2 \Phi) \quad \text{in } \Omega_T,
\end{equation*}
respectively, and \eqref{comp_init} yields 
\begin{equation*}
	\Phi^- \leq u \leq \Phi^+ \quad \text{on } \partial_p \Omega_T.
\end{equation*}
Since $\Phi_t^{\pm} < 0$ in $\Omega_T$, by \Cref{com_prin} and \Cref{com_prin2}, we have
\begin{equation*}
	\Phi^- \leq u \leq \Phi^+ \quad \text{on } \overline{\Omega_T},
\end{equation*}
which is the desired conclusion.
\end{proof}
\begin{remark} \label{rmk:positivity}
	By \Cref{lem:u<x}, any nonnegative solution $u \in C(\overline{\Omega_T})$ of \eqref{prob:main} satisfies $u>0$ in $\Omega_T$. This implies that the equation in \eqref{prob:main} is uniformly parabolic in $K_T$ for any $K \subset \joinrel \subset \Omega$.
\end{remark}
\Cref{lem:u<x} provides the following intuition: for a function $u$ that safisfies \eqref{eq:s-class}, we have
\begin{equation*}
	u(x,t) \approx \textnormal{dist}(x,\partial \Omega) \quad \text{near } \partial_s \Omega_T,
\end{equation*}
so roughly speaking, for the viscosity solution $u$ of \eqref{prob:main}, we have
\begin{equation*}
	u_t \approx [\textnormal{dist}(x,\partial \Omega)]^\gamma F(D^2 u,x,t) \quad \text{near } \partial_s \Omega_T.
\end{equation*}
In other words, it can be inferred that the boundary regularity for the fully nonlinear equation with degeneracy $[\textnormal{dist}(\cdot,\partial \Omega)]^\gamma$ is related to the boundary regularity for the equation in \eqref{prob:main}. Based on this fact, we first discuss boundary estimates for functions $u$ that satisfies
\begin{equation} \label{eq:dist_degen}
	\left\{\begin{aligned}
		u_t & \ge [\textnormal{dist}(x,\partial \Omega)]^\gamma \mathcal{M}_{\lambda,\Lambda}^-(D^2 u) && \text{in the viscosity sense in } \Omega_T \cap Q_{\rho}  \\
		u_t &\leq [\textnormal{dist}(x,\partial \Omega)]^\gamma \mathcal{M}_{\lambda, \Lambda}^+(D^2 u) && \text{in the viscosity sense in } \Omega_T \cap Q_{\rho} \\
		u&=g && \text{on } \partial_s \Omega_T \cap Q_{\rho}  .
	\end{aligned}\right.
\end{equation}

Throughout this section, we assume that $0 \in \partial \Omega$ and $e_n$ is the normal vector $\partial \Omega$ at $0$. We define the oscillation of the boundary suggested in \cite{LZ20,LZ22,Wan92b} as follows:
\begin{equation*}
	\osc_{B_\rho} \partial \Omega \coloneqq \sup_{x \in \partial \Omega \cap B_\rho} x_n -\inf_{x \in \partial\Omega\cap B_\rho} x_n \quad \text{for }\rho >0.
\end{equation*}

\begin{lemma} \label{lem1:gen}
	Let $\overline{\alpha}\in (0,1)$ and $C_*>1$ be as in \Cref{result:ly24}. Then, for any $\alpha \in (0,\overline{\alpha})$, there exists $\delta >0$ depending only on $n$, $\lambda$, $\Lambda$, $\gamma$, and $\alpha$ such that if $u$ is a function that satisfies \eqref{eq:dist_degen} with 
	\begin{equation*}
		\|u\|_{L^{\infty}(\Omega_T \cap Q_{\rho})} \leq 1, \quad 
		\|g\|_{L^{\infty}( \partial_s \Omega_T \cap Q_{\rho}  )} \leq \delta,  \quad \text{and} \quad
		\osc_{B_\rho} \partial \Omega \leq \delta,
	\end{equation*}
then there exists a constant $a \in \mathbb{R}$ such that
\begin{equation*} 
	\|u-ax_n\|_{L^{\infty}(\Omega_T\cap Q_{\eta})} \leq \eta^{1+\alpha}  
\end{equation*}
and $|a| \leq C_*$, where $\eta>0$ is a constant depending only on $n$, $\lambda$, $\Lambda$, $\gamma$, $\alpha$,  $\rho$, and $T$.
\end{lemma}

\begin{proof}
Suppose that the conclusion does not hold. Then, there exist sequences $\{u_k\}_{k=1}^\infty$, $\{g_k\}_{k=1}^\infty$, and $\{\Omega_k\}_{k=1}^\infty$ such that  
\begin{equation*}
	\|u_k\|_{L^{\infty}((\Omega_k)_T \cap Q_{\rho})} \leq 1, \quad 
	\|g_k\|_{L^{\infty}(\partial_s (\Omega_k)_T \cap Q_{\rho})} \leq 1/k, \quad
	\osc_{B_\rho} \partial \Omega_k \leq 1/k, 
\end{equation*}
and $u_k$ is a function that satisfies
	\begin{equation*}
		\left\{\begin{aligned}
			\partial_t u_k & \ge [\textnormal{dist}(x,\partial \Omega_k)]^\gamma \mathcal{M}_{\lambda,\Lambda}^-(D^2 u_k)  && \text{in the viscosity sense in } (\Omega_k)_T \cap Q_{\rho} \\
			\partial_t u_k & \leq [\textnormal{dist}(x,\partial \Omega_k)]^\gamma \mathcal{M}_{\lambda,\Lambda}^+(D^2 u_k)  && \text{in the viscosity sense in } (\Omega_k)_T \cap Q_{\rho} \\
			u_k&=g_k && \text{on } \partial_s (\Omega_k)_T \cap Q_{\rho}.
		\end{aligned}\right.
	\end{equation*}
Moreover, for any constant $a \in \mathbb{R}$ satisfying $|a| \leq C_*$, we have
\begin{align}\label{eq:contradiction}
	\|u_k-ax_n\|_{L^{\infty}((\Omega_k)_T\cap Q_{\eta})} >\eta^{1+\alpha},
\end{align}
where $\eta \in (0,1)$ will be determined later.

Note that $\{\tilde u_k\}_{k=1}^{\infty}$ is uniformly bounded and by the interior H\"older estimates, $\{\tilde u_k\}_{k=1}^{\infty}$ is equicontinuous in compact sets of $\bigcup_{k=1}^\infty \overline{(\Omega_k)_T}$, where 
\begin{equation*}
	\tilde u_k \coloneqq
	\begin{cases}
		u_k & \text{on } \overline{(\Omega_k)_T} \\
		0 & \text{otherwise.}
	\end{cases}
\end{equation*}
Thus, by Arzela-Ascoli theorem, there exist a subsequence $\{\tilde u_{k_j}\}_{j=1}^{\infty}$ of $\{ \tilde u_k\}_{k=1}^{\infty}$ and a limit function $\overline{u}$ such that $\tilde u_{k_j} \to \overline{u}$ uniformly in any compact set of  $\bigcup_{k=1}^\infty \overline{(\Omega_k)_T}$ as $j \to \infty$. Similarly, since the set $\{\chi_{\Omega_k}\}_{k=1}^{\infty}$ of characteristic functions is uniformly bounded and equicontinuous in any compact set of $\bigcup_{k=1}^\infty \Omega_k$, we also know that  there exist a subsequence $\{\Omega_{k_j}\}_{j=1}^{\infty}$ of $\{ \Omega_k\}_{k=1}^{\infty}$ and a limit domain $\Omega_\infty$ such that $\Omega_{k_j} \to \Omega_\infty$ as $j \to \infty$.

The stability theorem yields that 
\begin{equation*}
	\left\{\begin{aligned}
		\overline{u}_t & \ge [\textnormal{dist}(x,\partial \Omega_\infty)]^\gamma \mathcal{M}_{\lambda,\Lambda}^-(D^2 \overline{u}) && \text{in the viscosity sense in } (\Omega_\infty)_T\cap Q_{\rho}\\
		\overline{u}_t & \leq [\textnormal{dist}(x,\partial \Omega_\infty)]^\gamma \mathcal{M}_{\lambda,\Lambda}^+(D^2 \overline{u}) && \text{in the viscosity sense in } (\Omega_\infty)_T \cap Q_{\rho} \\
		\overline{u} &= 0 && \text{on } \partial_s (\Omega_\infty)_T \cap Q_{\rho}.
	\end{aligned}\right.
\end{equation*}
Furthermore, from $\osc_{B_\rho} \partial \Omega_{\infty} =\lim_{j\to\infty }\osc_{B_\rho} \partial \Omega_{k_j} = 0$, we have 
\begin{equation*}
	\textnormal{dist}(x,\partial \Omega_\infty) = x_n \quad \text{for all } x \in \Omega_\infty \cap B_{\rho}.
\end{equation*}
So, we can apply \Cref{result:ly24} to $\overline{u}$, we have
\begin{align*}
	|\overline{u} (x,t)- \overline{a} x_n | \leq C_*(|x|+\sqrt{|t|})^{1+\overline{\alpha}} \quad \text{for all $(x,t) \in \overline{(\Omega_\infty)_T \cap Q^+_{\rho/2}}$}
\end{align*}
and $|\overline{a}| \leq C_*$. For any $\alpha \in (0,\overline{\alpha})$, we take $\eta>0$ small enough so that $2^{1+\overline{\alpha}} \eta^{\overline{\alpha}-\alpha} C_*<1/2$. Then we observe that
\begin{align*}
	\|\overline{u}-\overline{a} x_n\|_{L^{\infty}(Q_{\eta}^+)} \leq 2^{1+\overline{\alpha}} \eta^{1+\overline{\alpha}} C_* <\eta^{1+\alpha}/2.
\end{align*}
On the other hand, by letting $k_j \to \infty$ in \eqref{eq:contradiction} with $a=\overline{a}$, we have
\begin{align*}
	\|\overline{u}-\overline{a} x_n\|_{L^{\infty}(Q_{\eta}^+)} \ge\eta^{1+\alpha},
\end{align*}
which leads to the contradiction.
\end{proof}

A standard technique to establish the boundary $C^{1,\alpha}$-estimates utilizes boundary Lipschitz estimates and Hopf principle to sandwich the solution $u$ between the two linear functions and applies it repeatedly through scaling to find a linear function that approximates $u$; we refer to \cite{SS14,Wan92b} for uniformly elliptic/parabolic equations and \cite{LY24} for degenerate/singular parabolic equations.
This technique is not applicable to 
\begin{equation*}
	u^{\gamma} \mathcal{M}^{-}_{\lambda, \Lambda}(D^2u)  \leq u_t \leq u^{\gamma} \mathcal{M}^{+}_{\lambda, \Lambda}(D^2u)  
\end{equation*}
because the coefficient function $u^\gamma$ of the Pucci's extremal operators $\mathcal{M}^{\pm}_{\lambda, \Lambda}$ is shifted in parallel by a linear function in the process of repeatedly applying this technique, and the class of the equation changes. 
One way of overcoming this difficulty is to establish regularity results for functions in the wider solution class. 

\begin{lemma} \label{lem2:gen}
	Let $\overline{\alpha}\in(0,1)$,  $C_*>1$, $\delta>0$, and $\eta >0$ be as in \Cref{lem1:gen}. Assume that  $\alpha \in(0,\overline{\alpha})$ and $u \in C(\overline{\Omega_T})$ is a nonnegative function that satisfies \eqref{eq:s-class} with $\|u\|_{L^{\infty}(\Omega_T)} \leq 1$ and $\|\partial \Omega\|_{C^{1,\alpha}(0)} \leq \delta /A$, where $A>2$ will be determined later. Then there exists a sequence $\{a_k\}_{k=-1}^{\infty}$
such that for all $k \geq 0$, we have 
\begin{align}\label{eq:lin1_bd}
	\|u-a_kx_n\|_{L^{\infty}(\Omega_T \cap Q_{\eta^k})} \leq \eta^{k(1+\alpha)}
\end{align}
and
\begin{align}\label{eq:diff_ak}
	|a_k-a_{k-1}| \leq  C_* \eta^{(k-1) \alpha}.
\end{align}
\end{lemma}
 
\begin{proof}
	 By \Cref{lem:u<x}, 
we obtain the following estimate:
\begin{equation*}
	C_1 \,\textnormal{dist}(x,\partial \Omega) \leq u(x,t) \leq C_2 \,\textnormal{dist}(x,\partial \Omega) \quad \text{for all } (x,t) \in \overline{\Omega_T}.
\end{equation*}
This implies that $u$ satisfies 
\begin{equation*} 
	[\textnormal{dist}(x,\partial \Omega)]^\gamma \mathcal{M}_{C_1^\gamma \lambda, C_2^\gamma \Lambda}^-(D^2 u) \leq u_t \leq [\textnormal{dist}(x,\partial \Omega)]^\gamma \mathcal{M}_{C_1^\gamma \lambda, C_2^\gamma\Lambda}^+(D^2 u) 
\end{equation*}
in the viscosity sense in $\Omega_T$. 

	The proof is by induction. For $k=0$, by setting $a_{-1}=a_0=0$, the conditions immediately hold. Suppose that the conclusion holds for $k \geq 0$. We claim that the conclusion also holds for $(k+1)$. For this purpose, let $r=\eta^k$, $y=x/r$, $s=t/r^{2-\gamma}$, and 
	\begin{align*}
		v(y,s) \coloneqq \frac{u(x,t)-a_kx_n}{r^{1+\alpha}}.
	\end{align*}
Then $v$ satisfies
	\begin{equation*}
		\left\{\begin{aligned}
		 	v_s &\ge [\textnormal{dist}(y,\partial \widetilde \Omega)]^\gamma \mathcal{M}^{-}_{C_1^\gamma \lambda, C_2^\gamma \Lambda}(D^2v) && \text{in the viscosity sense in } \widetilde{\Omega}_{\widetilde T} \cap Q_{1} \\
		 	v_s &\leq [\textnormal{dist}(y,\partial \widetilde \Omega)]^\gamma \mathcal{M}^{+}_{C_1^\gamma \lambda, C_2^\gamma \Lambda}(D^2v) && \text{in the viscosity sense in } \widetilde{\Omega}_{\widetilde T} \cap Q_{1} \\
			v &=g && \text{on } \partial_s \widetilde{\Omega}_{\widetilde T} \cap Q_{\rho} ,
		\end{aligned}\right.
	\end{equation*}
where $\widetilde\Omega \coloneqq r^{-1}\Omega = \{r^{-1} x: x \in \Omega\}$, $\widetilde T = T/r^{2-\gamma}$, and $g(y,s) \coloneqq -r^{-\alpha-1} a_k x_n $.

It immediately follows that $\|v\|_{L^{\infty}(\widetilde{\Omega}_{\widetilde T}\cap Q_1)} \leq 1$. From \eqref{eq:diff_ak}, there exists a constant $A>2$ depending only on $n$, $\lambda$, $\Lambda$, $\gamma$, $\alpha$,  and $T$ such that $|a_k| \leq A$ for all $k \ge1$ and hence we have
\begin{equation*}
	 |g(y,s)|  \leq r^{-\alpha-1} \times A \times  \|\partial \Omega\|_{C^{1,\alpha}(0)}|x'|^{1+\alpha} \leq \delta \quad \text{for all } (y,s) \in \partial_s \widetilde{\Omega}_{\widetilde T} \cap Q_{1} 
\end{equation*}
 Furthermore, we know that
\begin{equation*}
	\osc_{B_1} \partial \widetilde\Omega = \frac{1}{r} \osc_{B_{r}} \partial \Omega= 2 \|\partial \Omega\|_{C^{1,\alpha}(0)}r^{\alpha}  \leq \delta .
\end{equation*}
Thus, by applying \Cref{lem1:gen} to $v$, there exists a constant $\widetilde{a}\in\mathbb{R}$ such that 
\begin{align}\label{eq:induc1_bd}
	\|v-\widetilde{a} y_n\|_{L^{\infty}(\widetilde{\Omega}_{\widetilde T} \cap Q_{\eta})} \leq \eta^{1+\alpha}
\end{align}
and
\begin{align}\label{eq:induc2_bd}
	|\widetilde{a}| \leq C_*.
\end{align}
We now let $a_{k+1} \coloneqq a_k +r^{\alpha} \widetilde{a}$. Then \eqref{eq:induc2_bd} implies that
\begin{align*}
	|a_{k+1}-a_k|  \leq  C_* \eta^{k\alpha}.
\end{align*}
Finally, \eqref{eq:induc1_bd} shows that  
\begin{align*}
	|u(x,t)-a_{k+1}x_n| &=|u(x,t)-a_kx_n-r^{\alpha}\widetilde{a}x_n | \\
	&\leq r^{1+\alpha} |v(x/r,t/r^{2-\gamma})-\widetilde{a}(x_n/r)| \\
	&\leq \eta^{(k+1)(1+\alpha)}
\end{align*}
for all $(x,t) \in \Omega_T\cap Q_{\eta^{k+1}}$, as desired.
\end{proof}
\begin{proof}[Proof of \Cref{thm:bdry_c1a}]
	Without loss of generality, assume that $(x_0,t_0)=(0,0)$, $\|u\|_{L^{\infty}(\Omega_T)}\leq1$, and $e_n$ is the normal vector $\partial \Omega$ at $0$. We argue that the assumption $\|\partial \Omega\|_{C^{1,\alpha}(0)} \leq \delta /2$  in \Cref{lem2:gen} holds even after the appropriate reduction argument. In fact, since $\partial \Omega\in C^{1,\alpha}(0)$, we have
	\begin{align*}
		|x_n| \leq [\partial\Omega]_{C^{1,\alpha}(0)} |x'|^{1+\alpha} \quad \text{for all } x \in \partial \Omega \cap Q_1.
	\end{align*}
	This implies that 
	\begin{align*}
		|y_n| \leq \rho^\alpha [\partial \Omega]_{C^{1,\alpha}(0)} |y'|^{1+\alpha} \quad \text{for all } y \in \partial \widetilde{\Omega} \cap Q_1,
	\end{align*}
where $\widetilde{\Omega} \coloneqq \rho^{-1} \Omega =\{ \rho^{-1} x: x \in \Omega\}$. By taking $\rho>0$ small enough, we may assume that $\|\partial \Omega\|_{C^{1,\alpha}(0)} \leq \delta /2$.

Therefore, by \Cref{lem2:gen}, there exists a sequence $\{a_k\}_{k=-1}^{\infty}$ satisfying \eqref{eq:lin1_bd}. According to the standard argument, we obtain limit $a$ such that $a_k \to a$ satisfying
\begin{align*}
	|a_k-a| \leq C\eta^{k \alpha }.
\end{align*}

Finally, for any $(x,t) \in \Omega_T \cap  Q_1$, there exists $j \geq 0$ such that $\eta^{j+1} \leq \max\{ |x|, \sqrt{|t|} \}< \eta^j$, and hence we conclude that
\begin{equation*}
	|u(x,t)-ax_n| \leq |u(x,t)-a_j x_n|+|a_j -a| |x_n| \leq C(|x|+\sqrt{|t|})^{1+\alpha},
\end{equation*}
which implies $u\in C^{1, \alpha}(0,0)$.
\end{proof}
\begin{remark} \label{rmk:c1a_dist_degen}
Let $f\in C(\Omega_T)\cap L^\infty(\Omega_T)$, $g\in C^{1,\alpha}(0,0)$, and $\partial \Omega \in C^{1,\alpha}(0)$. Then, the results of \Cref{lem1:gen}, \Cref{lem2:gen}, and \Cref{thm:bdry_c1a}  also hold for functions that satisfies
\begin{equation*}
	\left\{\begin{aligned}
		u_t & \ge [\textnormal{dist}(x,\partial \Omega)]^\gamma \mathcal{M}_{\lambda,\Lambda}^-(D^2 u) - \|f\|_{L^\infty(\Omega_T)} && \text{in the viscosity sense in } \Omega_T \cap Q_{\rho}  \\
		u_t &\leq [\textnormal{dist}(x,\partial \Omega)]^\gamma \mathcal{M}_{\lambda, \Lambda}^+(D^2 u) + \|f\|_{L^\infty(\Omega_T)} && \text{in the viscosity sense in } \Omega_T \cap Q_{\rho} \\
		u&=g && \text{on } \partial_s \Omega_T \cap Q_{\rho} 
	\end{aligned}\right.
\end{equation*}
by the same proof and hence we obtain
\begin{equation*}
	|u(x,t)-ax_n| \leq C(\|u\|_{L^\infty(\Omega_T)} + \|f\|_{L^\infty(\Omega_T)} + \|g\|_{C^{1,\alpha}(0,0)})(|x|+\sqrt{|t|})^{1+\alpha}
\end{equation*}
for all $(x,t) \in \Omega_T \cap  Q_1$, where $C>0$ is a constant depending only on $n$, $\lambda$, $\Lambda$, $\gamma$, $\alpha$, and $\|\partial \Omega\|_{C^{1,\alpha}(0)}$.
\end{remark}
%
%
\section{Global $C^{2,\alpha}$-estimates} \label{sec:dirichlet}
In this section, we prove the global $C^{2,\alpha}$-estimates for viscosity solutions of \eqref{prob:main}. From the intuition that $u(x,t)\approx \textnormal{dist}(x,\partial \Omega)$ near $\partial_s \Omega_T$, we were able to convert the $u^\gamma$ degeneracy into the $[\textnormal{dist}(\cdot,\partial \Omega)]^\gamma$ degeneracy, which made it possible to apply the Schauder-type estimates. However, this intuition only works up to the $C^{1,\alpha}$-estimates for \eqref{prob:main}, so a more accurate approximation is required to obtain the $C^{2,\alpha}$-estimates. Roughly speaking, to obtain the $C^{2,\alpha}$-estimates, we need an approximation of the following form:
\begin{equation*}
	\frac{u(x,t)}{\textnormal{dist}(x,\partial \Omega)} -  a(x',t) \approx [ \textnormal{dist}(x,\partial \Omega)]^{\alpha} \quad \text{near $\partial_s \Omega_T$,}
\end{equation*}
which will be proven in \Cref{lem:ly24_gen}.

Throughout this section, we assume that the nonlinear operator $F$ satisfies  \textnormal{\ref{F1}}-\textnormal{\ref{F3}},  $0 \in \partial \Omega$, and $e_n$ is the normal vector $\partial \Omega$ at $0$. 
\subsection{Boundary $C^{2,\alpha}$-estimates}
We first discuss boundary estimates for viscosity solutions of
\begin{equation} \label{eq:dist_degen_c2a}
	\left\{\begin{aligned}
		u_t & = [\textnormal{dist}(x,\partial \Omega)]^\gamma  F(D^2 u,x,t) && \text{in } \Omega_T \cap Q_{\rho}  \\
		u&=g && \text{on } \partial_s \Omega_T \cap Q_{\rho}.
	\end{aligned}\right.
\end{equation}
\begin{lemma} \label{lem1:gen_c2a}
	Let $C_\star>1$ be as in \Cref{result:ly24-2}. For any $\alpha \in (0,\overline{\alpha})$ with $\alpha < 1-\gamma$, there exists $\delta >0$ depending only on $n$, $\lambda$, $\Lambda$, $\gamma$, and $\alpha$ such that if $u$ is a viscosity solution of \eqref{eq:dist_degen_c2a} with
\begin{align*}
	&\|u\|_{L^{\infty}( \Omega_T \cap Q_{\rho} )} \leq 1, \quad
	u(0,0) =0 = |Du(0,0)|, \quad
	\|\partial \Omega\|_{C^{1,\alpha}(0)} \leq \delta, \\ 
	&\|\beta^1\|_{L^{\infty}(\Omega_T \cap Q_{\rho} )}\leq \delta, \quad 
	\|\beta^2\|_{L^{\infty}(\Omega_T \cap Q_{\rho} )}\leq \delta, \quad \text{and} \quad
	\|g\|_{C^{1,\alpha}(0,0)} \leq \delta,
\end{align*}
then there exists a polynomial $P(x)  = \sum_{i=1}^n a_i x_i x_n$ such that $ \sum_{i=1}^n |a_i| \leq C_\star$,  $F(D^2 P,0,0)=0$, and
\begin{equation*}
	\|u- P \|_{L^{\infty}(\Omega_T \cap Q_{\eta})} \leq \eta^{2+\alpha},
\end{equation*}
where $\eta>0$ is a constant depending only on $n$, $\lambda$, $\Lambda$, $\gamma$, $\alpha$,  $\rho$, and $T$.
\end{lemma}

\begin{proof}
Suppose that the conclusion does not hold. Then, there exist sequences $\{u_k\}_{k=1}^\infty$, $\{g_k\}_{k=1}^\infty$, $\{\beta_k^1\}_{k=1}^\infty$, $\{\beta_k^2\}_{k=1}^\infty$, $\{F_k\}_{k=1}^\infty$, and $\{\Omega_k\}_{k=1}^\infty$ such that 
\begin{align*}
	& \|u_k\|_{L^{\infty}((\Omega_k)_T\cap Q_{\rho})}  \leq 1, \quad
	u_k(0,0) =0 =|Du_k(0,0)|, \quad
	\|\partial\Omega_k \|_{C^{1,\alpha}(0)} \leq 1/k, \\
	&\|\beta_k^1\|_{L^{\infty}((\Omega_k)_T\cap Q_{\rho} )} \leq 1/k, \quad
	\|\beta_k^2\|_{L^{\infty}((\Omega_k)_T\cap Q_{\rho} )}\leq 1/k, \quad
	\|g_k\|_{C^{1,\alpha}(0,0)}  \leq 1/k,  
\end{align*}
and $u_k$ is a viscosity solution of 
	\begin{equation*}
		\left\{\begin{aligned}
			\partial_ t u_k &=  [\textnormal{dist}(x,\partial \Omega_k)]^\gamma  F_k(D^2 u_k,x,t) && \text{in } (\Omega_k)_T\cap Q_{\rho} \\
			u_k &= g_k && \text{on } \partial_s(\Omega_k)_T \cap Q_{\rho}.
		\end{aligned}\right.
	\end{equation*}
Moreover, for any polynomial $P(x)  = \sum_{i=1}^n a_i x_i x_n$ satisfying $ \sum_{i=1}^n |a_i|  \leq C_\star$ and $F_k(D^2 P,0,0)=0$, we have
\begin{align}\label{eq:cont_c2a}
	\|u_k- P\|_{L^{\infty}((\Omega_k)_T \cap Q_{\eta})} >\eta^{2+\alpha},
\end{align}
where $\eta \in (0,1)$ will be determined later.

As in the proof of \Cref{lem1:gen}, there exist subsequences $\{u_{k_j}\}_{j=1}^{\infty}$, $\{\Omega_{k_j}\}_{j=1}^{\infty}$ and limits $\overline{u}$, $\Omega_\infty$ such that $u_{k_j} \to \overline{u}$ uniformly in any compact sets of $\bigcup_{k=1}^\infty \overline{(\Omega_k)_T} $ as $j \to \infty$ and $\lim_{j\to\infty} \Omega_{k_j} = \Omega_\infty$. 
Here, we omit the tilde symbol for simplicity.
Since $F_k$ is Lipschitz continuous in $M$ with a uniform Lipschitz constant, $\{F_k\}_{k=1}^{\infty}$ is also uniformly bounded and equicontinuous in any compact sets of $\mathcal{S}^n$. Then there exist a subsequence $\{F_{k_j}\}_{j=1}^{\infty}$ and a limit operator $\overline{F}$ such that $F_{k_j}(\cdot,0,0) \to \overline{F}$ uniformly in any compact sets of $\mathcal{S}^n$ as $j \to \infty$. It follows that 
\begin{align*}
	|F_{k_j}(M,x,t) - \overline{F}(M)| &\leq |F_{k_j}(M,x,t) - F_{k_j}(M,0,0)| + |F_{k_j}(M,0,0)- \overline{F}(M)| \\
	&\leq \beta_{k_j}^1(x,t) \|M\| + \beta_{k_j}^2(x,t)   + |F_{k_j}(M,0,0)- \overline{F}(M)| \\
	&\leq \frac{1}{k_j} (\|M\| +1)  + |F_{k_j}(M,0,0)- \overline{F}(M)| 
\end{align*}
for all $M\in\mathcal{S}^n$ and $(x,t) \in (\Omega_k)_T\cap Q_{\rho}$. Since the right-hand side above tends to zero uniformly in any compact sets in $\mathcal{S}^n$ as $j \to \infty$, the stability theorem yields that 
\begin{equation*}
	\left\{\begin{aligned}
		\overline{u}_t &= x_n^\gamma \overline{F}(D^2 \overline{u}) && \text{in } (\Omega_\infty)_T \cap Q_{\rho}^+  \\
		\overline{u} &= 0 && \text{on } \partial_s(\Omega_\infty)_T \cap Q_{\rho}^+.
	\end{aligned}\right.
\end{equation*}

Since $u_{k_j}(0,0) =0=|Du_{k_j}(0,0)|$ and $\partial \Omega_k \in C^{1,\alpha}(0)$, by applying \Cref{rmk:c1a_dist_degen} to $u_{k_j}$, we have
\begin{equation} \label{kth_r1a}
	\|u_{k_j}\|_{L^{\infty}((\Omega_{k_j})_T \cap Q_r)} \leq C r^{1+\alpha} \quad \text{for all } r \in (0,1),
\end{equation}
where $C>0$ is a constant depending only on $n$, $\lambda$, $\Lambda$, $\gamma$, $\alpha$, and $\|\partial \Omega_{k_j}\|_{C^{1,\alpha}(0)}$. Since $u_{k_j} \to \overline{u}$ uniformly, by letting $j \to \infty$ in \eqref{kth_r1a}, we obtain
\begin{equation*}
	\|\overline{u}\|_{L^{\infty}((\Omega_\infty)_T \cap Q_r)} \leq C r^{1+\alpha} \quad \text{for all } r \in (0,1),
\end{equation*}
which implies that $\overline{u}(0,0)=0=|D\overline{u}(0,0)|$. 
Thus, by \Cref{result:ly24-2}, there exists a polynomial $\overline{P}(x) = \sum_{i=1}^n \overline{a}_i x_i x_n$ such that $  \sum_{i=1}^n |\overline{a}_i| \leq C_\star$, $\overline{F}(D^2\overline{P})=0$, and
\begin{equation*}
	|\overline{u}(x,t)-\overline{P}(x)| \leq C_\star |x|^{1+\frac{\alpha+1-\gamma}{2}} x_n \quad \text{for all } (x,t) \in \overline{(\Omega_\infty)_T \cap Q_{\rho/2}^+}.
\end{equation*}

For any $\alpha \in (0,\overline{\alpha})$ with $\alpha<1-\gamma$, we take $\eta>0$ small enough so that $\eta^{\frac{1-\gamma-\alpha}{2}} C_\star<1/2$. Then we obtain
\begin{align*}
	\|\overline{u}-\overline{P}\|_{L^{\infty}(Q_{\eta}^+)} <\eta^{2+\alpha}/2.
\end{align*}
On the other hand, there exists $\{\theta_{k_j} \}_{j=1}^\infty$ with $\lim_{j\to\infty}\theta_{k_j} = 0$ and $|\theta_{k_j}| \leq 1$ such that $F_{k_j}(D^2\overline{P}+\theta_{k_j} I_n,0,0) = 0$, since $\lim_{j\to \infty}F_{k_j}(D^2 \overline{P},0,0) = \overline{F}(D^2\overline{P}) = 0$. From \eqref{eq:cont_c2a}, we have
 \begin{equation*}
	\|u_{k_j}- \overline{P} - \theta_{k_j} |x|^2/2\|_{L^{\infty}((\Omega_{k_j})_T \cap Q_{\eta})} >\eta^{2+\alpha},
\end{equation*}
and hence by letting $k_j \to \infty$, we obtain
\begin{align*}
	\|\overline{u}-\overline{P}\|_{L^{\infty}(Q_{\eta}^+)} \ge\eta^{2+\alpha},
\end{align*}
which leads to the contradiction. 
\end{proof}
\begin{lemma} \label{lem2:gen_c2a}
	Let  $C_\star>1$, $\delta>0$, and $\eta >0$ be as in \Cref{lem1:gen_c2a}. Suppose that  $\alpha \in(0,\overline{\alpha})$ with $\alpha< 1-\gamma$ and $u$ is a nonnegative viscosity solution of \eqref{eq:dist_degen_c2a} with 
	\begin{equation*}
		\|u\|_{L^{\infty}(\Omega_T)} \leq 1, \quad 
		u(0,0)=0, \quad
		Du(0,0)=0, \quad
		\|\partial \Omega\|_{C^{1,\alpha}(0)} \leq \frac{\delta}{2\sqrt{n}K},
	\end{equation*}
and 
\begin{align} 
	\beta^1(x,t) &\leq \frac{\delta}{K 2^{1+\alpha}} (|x|+\sqrt{|t|})^{\alpha} \quad \text{for all } (x,t) \in \overline{\Omega_T\cap Q_1} , \label{beta_hol1} \\
	\beta^2(x,t) &\leq \frac{\delta}{2^{1+\alpha}} (|x|+\sqrt{|t|})^{\alpha} \quad \text{for all } (x,t) \in \overline{\Omega_T\cap Q_1} , \label{beta_hol2} 
\end{align}  
and
\begin{equation*}
	|g(x,t)| \leq \frac{\delta}{2} (|x|+\sqrt{|t|})^{2+\alpha} \quad \text{for all } (x,t) \in \partial_s \Omega_T \cap Q_1  , \label{g_c2a} 
\end{equation*}
where  $K>1$ is a contant depending only on $n$, $\lambda$, $\Lambda$, $\gamma$, and $\alpha$. 
	Then there exists a sequence $\{P_k\}_{k=-1}^{\infty}$ of the form $P_k(x)= \sum_{i=1}^n a_i^k x_i x_n$
such that for all $k \geq 0$, we have 
\begin{equation*}
	\|u-P_k\|_{L^{\infty}(\Omega_T \cap Q_{\eta^k})} \leq \eta^{k(2+\alpha)}, \quad   F(D^2P_k,0,0)=0,
\end{equation*}
and
\begin{equation}\label{eq:diff_ak2}
	\sum_{i=1}^n |a^k_i-a^{k-1}_i| \leq  C_\star \eta^{(k-1)\alpha}.
\end{equation}
\end{lemma}

\begin{proof}
	The proof is by induction. For $k=0$, by setting $P_{-1}=P_0\equiv0$, the conditions immediately hold. Suppose that the conclusion holds for $k \geq 0$. We claim that the conclusion also holds for $(k+1)$.
	
	For this purpose, let $r=\eta^k$, $y=x/r$, $s=t/r^{2-\gamma}$, and 
	\begin{align*}
		v(y,s) \coloneqq \frac{u(x,t)-P_k(x)}{r^{2+\alpha}}.
	\end{align*}
Then $v$ is a viscosity solution of
	\begin{equation*}
		\left\{\begin{aligned}
			v_s &= [\textnormal{dist}(y,\partial \widetilde \Omega)]^\gamma \widetilde{F}(D^2 v,y,s)  && \text{in } \widetilde{\Omega}_{\widetilde T}  \cap Q_{1} \\
			v &= \widetilde{g} && \text{on } \partial_s \widetilde{\Omega}_{\widetilde T}  \cap Q_{1},
		\end{aligned}\right.
	\end{equation*}
where $\widetilde\Omega \coloneqq r^{-1}\Omega = \{r^{-1} x: x \in \Omega\}$, $\widetilde T \coloneqq T/r^{2-\gamma}$,  
\begin{equation*}
	\widetilde{F}(M,y,s)\coloneqq r^{-\alpha}  F(r^\alpha M + D^2 P_k,x,t) \quad \text{and} \quad  \widetilde{g}(y,s)\coloneqq \frac{g(x,t)-P_k(x)}{r^{2+\alpha} }.
\end{equation*}
Then it immediately follows that 
\begin{equation*}
	\|v\|_{L^{\infty}(\widetilde{\Omega}_{\widetilde T} \cap Q_{1})} \leq 1,\quad v(0,0)=0=|Dv(0,0)|
\end{equation*}
and
\begin{equation*}
	 \|\partial \widetilde\Omega\|_{C^{1,\alpha}(0)} \leq r^{\alpha} \|\partial \Omega\|_{C^{1,\alpha}(0)} \leq \delta.
\end{equation*}
From \eqref{eq:diff_ak2}, there exists a constant $K>1$ depending only on $n$, $\lambda$, $\Lambda$, $\gamma$, and $\alpha$ such that $\|D^2 P_k\|  \leq K$ for all $k \ge1$ and hence we have
\begin{align*}
	 |g(y,s)| 
	 &\leq r^{-2-\alpha} \left( \frac{\delta}{2} r^{2+\alpha} (|y|+\sqrt{|s|})^{2+\alpha}+ \sqrt{n}K \|\partial \Omega\|_{C^{1,\alpha}(0)} |x| |x'|^{1+\alpha}  \right)\\
	 &\leq \delta (|y|+\sqrt{|s|})^{2+\alpha} \quad \text{for all } (y,s) \in \partial_s \widetilde{\Omega}_{\widetilde T}  \cap Q_{1} . 
\end{align*}
By defining $\widetilde{\beta}^1$ and $\widetilde{\beta}^2$ as follows
\begin{equation*}
	\widetilde{\beta}^1(y,s) \coloneqq  \beta^1(x,t)    \quad \text{and} \quad
	\widetilde{\beta}^2(y,s) \coloneqq  r^{-\alpha} K \beta^1(x,t) +r^{-\alpha} \beta^2(x,t) ,
\end{equation*}
we see 
\begin{align*}
	|\widetilde{F}(M,y,s)- \widetilde{F}(M,0,0)|
	&\leq\widetilde{\beta}^1(y,s) \|M\| + \widetilde{\beta}^2(y,s).
\end{align*}
Combining \eqref{beta_hol1} and \eqref{beta_hol2} gives 
\begin{equation*}
	\|\widetilde{\beta}^1\|_{L^{\infty}(\widetilde{\Omega}_{\widetilde T}  \cap Q_1  )}\leq \delta \quad \text{and} \quad \|\widetilde{\beta}^2\|_{L^{\infty}(\widetilde{\Omega}_{\widetilde T} \cap Q_1)}\leq \delta. 
\end{equation*}
By \Cref{lem1:gen_c2a} for $v$, there exists a polynomial $\widetilde{P}(y)= \sum_{i=1}^n \widetilde{a}_i y_i y_n$ such that 
\begin{align}\label{eq:ind1_bd}
	\|v-\widetilde{P}\|_{L^{\infty}(\widetilde{\Omega}_{\widetilde{T}} \cap Q_{\eta})} \leq \eta^{2+\alpha},
\end{align}
\begin{align}\label{eq:ind2_bd}
	\sum_{i=1}^n |\widetilde{a}_i| \leq C_\star ,\quad \text{and} \quad \widetilde{F}(D^2 \widetilde{P},0,0)=0.
\end{align}
We now let $P_{k+1}(x) \coloneqq P_k(x) +r^{2+\alpha} \widetilde{P}(x/r)$. Then \eqref{eq:ind2_bd} shows that
\begin{align*}
	\sum_{i=1}^n |a^{k+1}_i-a^k_i|  \leq  C_\star \eta^{k \alpha} \quad \text{and} \quad F(D^2P_{k+1},0,0) = 0.
\end{align*}
Finally, \eqref{eq:ind1_bd} implies   
\begin{align*}
	|u(x,t)-P_{k+1}(x)| &=|u(x,t)- P_k(x)-r^{2+\alpha}\widetilde{P}(x/r) | \\
	&\leq r^{2+\alpha} |v(x/r,t/r^{2-\gamma})-\widetilde{P}(x/r)| \\
	&\leq \eta^{(k+1)(2+\alpha)}
\end{align*}
for all $(x,t) \in \Omega_T \cap Q_{\eta^{k+1}}$.
\end{proof}

In the regularity theory of nonlinear PDE problems, the Schauder-type estimates is mainly used to establish $C^{2,\alpha}$-regularity of solutions. However, due to the structure of the product $u^\gamma$ and $F$ in \eqref{prob:main}, it is impossible to apply Schauder-type estimates to the equation in \eqref{prob:main}. Even if we consider a wider solution class, such as when obtaining the $C^{1,\alpha}$-estimates, this difficulty cannot be resolved because the boundary regularity of functions in the solution class is at most $C^{1,\alpha}$. To solve this, we try to convert the structure of the product $u^\gamma$ and $F$ into a H\"older continuity of fully nonlinear operator. To be precise, we convert the equation in \eqref{prob:main} as follows:
\begin{equation} \label{eq:convert}
	u_t = u^\gamma F(D^2 u,x,t) = [\textnormal{dist}(x,\partial \Omega)]^\gamma \widetilde{F}(D^2u,x,t) \quad \text{in } \Omega_T,
\end{equation}
where 
\begin{equation*}
	h(x,t) \coloneqq \frac{u(x,t)}{\textnormal{dist}(x,\partial \Omega)}  \quad \text{and} \quad \widetilde{F}(D^2u,x,t) \coloneqq h^\gamma F(D^2u,x,t).
\end{equation*}
According to \Cref{lem:u<x}, we have $C_1<h<C_2$ in $\overline{\Omega_T}$, so $\widetilde{F}$ satisfies \ref{F1}. Therefore if we have $h \in C^{\alpha_0}(\overline{\Omega_T})$ for some $\alpha_0 \in (0,1)$, Schauder-type estimates can be applied to the equation in \eqref{prob:main}.
\begin{lemma}\label{lem:ly24_gen}
	Let $\alpha_0 \in (0,\overline{\alpha})$ with $\alpha_0 \leq 1-\gamma$ and $\partial \Omega \in C^2$.  Assume that $u \in C(\overline{\Omega_T})$ is a nonnegative viscosity solution of \eqref{prob:main}. Then there exists a constant $a \in \mathbb{R}$ such that
\begin{equation*}
	\left| \frac{u(x,t)}{\textnormal{dist}(x,\partial \Omega)}-a \right | \leq C_0(|x| + \sqrt{|t|})^{\alpha_0} \quad \text{for all } (x,t) \in \overline{\Omega_T \cap Q_{\rho}}
\end{equation*}
and $|a| \leq C_0$, where $C_0>1$ is a constant depending only on $n$, $\lambda$, $\Lambda$, $\gamma$, $\alpha_0$, $ \|u\|_{L^{\infty}(\Omega_T)}$, and $\|\partial \Omega\|_{C^2}$.
\end{lemma}

\begin{proof}
By \Cref{lem:u<x}, there exist constants $C_1>0$ and $C_2>0$ such that
\begin{equation} \label{lip_u}
	C_1 \, \textnormal{dist}(x,\partial \Omega) \leq u(x,t) \leq C_2 \, \textnormal{dist}(x,\partial \Omega) \quad \text{for all } (x,t) \in \overline{\Omega_T}.
\end{equation}
Take a coordinate chart $\varphi:\overline{B_1^+} \to \overline{\Omega} \cap  \overline{B_r}$ that flattens $\partial \Omega$ such that 
\begin{equation} \label{flattening}
	\textnormal{dist}( \varphi(x),\partial \Omega)= x_n,  \quad
	\varphi(0)=0, \quad \text{and} \quad 
	\varphi(\partial B_1^+ \cap \{x_n = 0\}) \subset \partial \Omega \cap \overline{B_r}.
\end{equation}
Then, $v(x,t)=u(\varphi(x),t)$ is a viscosity solution of 
\begin{equation*}
\left\{\begin{aligned}
	v_t &= v^\gamma \widetilde{F}(D^2v, Dv,x,t) && \text{in } (B_1^+)_T \\
	v&=0 && \text{on } \{x_n=0\},
\end{aligned}\right.
\end{equation*}	
where 
\begin{equation*}
	\widetilde{F}(D^2v, Dv,x,t)= F((D\psi^T \circ \varphi) (D^2 v) (D\psi \circ \varphi) + (Dv)(D^2 \psi \circ \varphi) ,\varphi,t) \quad \text{for } \psi\coloneqq\varphi^{-1}.
\end{equation*}
Furthermore, \eqref{lip_u} implies that
\begin{equation} \label{lip_v}
	C_1 x_n \leq v(x,t) \leq C_2 x_n \quad \text{for all } (x,t) \in \overline{(\psi(\Omega))_T}.
\end{equation}

We know from \Cref{rmk:positivity} that the equation in \eqref{prob:main} is uniformly parabolic in $K_T$ for any domain $K \subset \joinrel \subset \Omega$ and $u$ satisfies
\begin{equation*}
	\mathcal{M}^{-}_{m \lambda, M \Lambda}(D^2u) \leq u_t \leq  \mathcal{M}^{+}_{m \lambda, M \Lambda}(D^2u) 
\end{equation*}
in the viscosity sense in $K_T$, where $m\coloneqq \inf_{K_T} u^\gamma>0$ and $M\coloneqq \sup_{K_T} u^\gamma$. By Krylov--Safonov theory for fully nonlinear uniformly parabolic equations, we have $u \in C^{\alpha}_{\textnormal{loc}}(\Omega_T)$ for some $\alpha \in (0,1)$. Since $F$ satisfies \ref{F2} and \ref{F3}, we have $u \in C_{\textnormal{loc}}^{2,\alpha\gamma}(\Omega_T)$ from the interior regularity results for fully nonlinear uniformly parabolic equations. Thus, we can treat $Dv$ as just an $(x,t)$-variable. Moreover, $\widetilde F$ satisfies \ref{F1}, we can see that $v$ satisfies  
\begin{equation} \label{class_v}
	v^\gamma \left( \mathcal{M}_{\lambda,\Lambda}^-(D^2 v) + \widetilde{F}(O_n,Dv,x,t) \right) \leq v_t \leq v^\gamma \left( \mathcal{M}_{\lambda,\Lambda}^+(D^2 v) + \widetilde{F}(O_n,Dv,x,t) \right)
\end{equation}
in the viscosity sense in $ (B_1^+)_T $. Combining \eqref{lip_v} and \eqref{class_v} gives 
\begin{equation*}
	x_n^\gamma \mathcal{M}_{C_1^\gamma \lambda, C_2^\gamma \Lambda}^-(D^2 v) +  \widetilde{f} \leq v_t \leq x_n^\gamma \mathcal{M}_{C_1^\gamma \lambda, C_2^\gamma\Lambda}^+(D^2 v) +   \widetilde{f} 
\end{equation*}
in the viscosity sense in $(\psi(\Omega))_T$, where $\widetilde{f}(x,t) \coloneqq x_n^\gamma \widetilde{F}(O_n,Dv,x,t)$. 
By applying \Cref{rem:xn1a-growth} to $v$, we have $v \in C^{1,\alpha_0}(\overline{Q_{1/2}^+})$ and 
\begin{equation} \label{est:v/x}
	\left| \frac{v(x,t)}{x_n}- v_n(x',0,t) \right | \leq C_\diamond (\|v\|_{L^{\infty}((B_1^+)_T)} + \|\widetilde{f}\|_{L^{\infty}((B_1^+)_T)}) x_n^{\alpha_0} 
\end{equation}
for all $(x,t) \in \overline{(B_{1/2}^+)_{T/2}}$. Combining \eqref{flattening} and \eqref{est:v/x} gives 
\begin{align*}
	\left| \frac{u(x,t)}{\textnormal{dist}(x,\partial \Omega)}- v_n(0,0) \right | 
	&\leq C_0 \left(\|\psi\|_{C^1(\overline{(\varphi(B_1^+))_T})} |x| + \sqrt{|t|}\right)^{\alpha_0} 
	\leq C_0  ( |x| + \sqrt{|t|})^{\alpha_0} 
\end{align*}
for all $ (x,t) \in \overline{(\varphi(B_{1/2}^+))_{T/2}}$.
\end{proof}
\begin{proof}[Proof of \Cref{thm:bdry_c2a}]
	Without loss of generality, assume that $(x_0,t_0)=(0,0)$ and $e_n$ is the normal vector $\partial \Omega$ at $0$. We claim that the assumptions in \Cref{lem2:gen_c2a}  hold after the appropriate reduction argument. 

Since $\partial \Omega \in C^{2,\alpha}(0)$, there exists a polynomial $P_{\partial \Omega}(x')$ with $\deg P_{\partial \Omega} \leq 2$, $P_{\partial \Omega}(0')=0$, and $DP_{\partial \Omega}(0')=0$ such that 
\begin{equation} \label{p_om_c2a}
	|x_n - P_{\partial \Omega}(x')| \leq \|\partial \Omega\|_{C^{2,\alpha}(0)} |x'|^{2+\alpha} \quad \text{for all }  x \in \partial \Omega \cap B_1.
\end{equation}

By \Cref{thm:bdry_c1a}, $Du(0,0)$ is well-defined and the function 
$$\widetilde{u}(x,t) \coloneqq u(x,t) - u_n(0,0)\big(x_n-P_{\partial \Omega}(x')\big)$$ 
is a viscosity solution of 
\begin{equation*} 
\left\{\begin{aligned}
	\widetilde{u}_t &= [\textnormal{dist}(x,\partial \Omega)]^\gamma \widetilde{F}(D^2 \widetilde{u},x,t) && \text{in } \Omega_T \\
	\widetilde{u} &=g && \text{on } \partial_s \Omega_T ,
\end{aligned}\right.
\end{equation*}
where $g(x,t) \coloneqq  - u_n(0,0)\big(x_n-P_{\partial \Omega}(x')\big)$,
\begin{equation*}
	h(x,t) \coloneqq \frac{u(x,t)}{\textnormal{dist}(x,\partial \Omega)} \quad  \text{and} \quad  \widetilde{F}(M,x,t) \coloneqq h^\gamma F(M -u_n(0,0) D^2 P_{\partial \Omega},x,t).
\end{equation*}
From \eqref{p_om_c2a}, we know that
\begin{equation*}
	|g(x,t)| \leq  |u_n(0,0)| \times \|\partial \Omega\|_{C^{2,\alpha}(0)}(|x| + \sqrt{|t|})^{2+\alpha} \quad \text{for all }  (x,t) \in \partial_s \Omega_T \cap Q_1.
\end{equation*}
Define
\begin{align*}
	\widetilde{\beta}^1(y,s) &\coloneqq C_2^\gamma \beta^1(x,t) + |h(x,t) - h(0,0)|^\gamma  \quad \text{and} \\
	\widetilde{\beta}^2(y,s) &\coloneqq   C_2^\gamma |u_n(0,0)| \times \|\partial \Omega\|_{C^{2,\alpha}(0)} \beta^1(x,t) + C_2^\gamma  \beta^2(x,t) \\
	&\quad+ |u_n(0,0)| \times \|\partial \Omega\|_{C^{2,\alpha}(0)} |h(x,t) - h(0,0)|^\gamma ,
\end{align*}	
where $C_2>0$ is the constant given in \Cref{lem:u<x}. By \Cref{lem:ly24_gen}, we have
\begin{equation*}
	|h(x,t) -h(0,0)| \leq C_0 \|u\|_{L^\infty(\Omega_T)} (|x| +\sqrt{|t|})^{\alpha_0} \quad \text{for all } (x,t) \in \overline{\Omega_T \cap Q_1}
\end{equation*}
for $\alpha_0 \coloneqq (\alpha + \gamma \min \{\overline{\alpha},1-\gamma\})/(2\gamma)$ and hence $\widetilde{\beta}^1, \widetilde{\beta}^2 \in C^\alpha(0,0)$ and $\widetilde{\beta}^1(0,0)=0=\widetilde{\beta}^2(0,0)$. Furthermore, we obtain
\begin{align*}
	&|\widetilde{F}(M,y,s)- \widetilde{F}(M,0,0)| \\
	&\quad \leq h(x,t)^\gamma  |F( M -u_n(0,0) D^2 P_{\partial \Omega},x,t)- F( M -u_n(0,0) D^2 P_{\partial \Omega},0,0)| \\
	&\qquad +  |h(x,t) - h(0,0)|^\gamma |F( M -u_n(0,0) D^2 P_{\partial \Omega},0,0)|\\
	&\quad \leq C_2^\gamma \left(\beta^1(x,t) \|M\| + |u_n(0,0)| \times \|\partial \Omega\|_{C^{2,\alpha}(0)}\|  \beta^1(x,t) + \beta^2(x,t) \right) \\
	&\qquad +  |h(x,t) - h(0,0)|^\gamma(\|M\| + |u_n(0,0)| \times \|\partial \Omega\|_{C^{2,\alpha}(0)}\| )\\
	&\quad\leq\widetilde{\beta}^1(y,s) \|M\| + \widetilde{\beta}^2(y,s).
\end{align*}
Let
	\begin{align*}
		N&\coloneqq \|\widetilde{u}\|_{L^{\infty}(\Omega_T)} + 2^{1+\alpha} \delta^{-1} [\widetilde{\beta}^2]_{C^{\alpha}(0,0)} + 2\delta^{-1} \|g\|_{C^{2,\alpha}(0,0)}+ |\widetilde{F}(O_n,0,0)|,
	\end{align*}
by considering $\widetilde{u}/N$ and $g/N$, let us assume without loss of generality that $\|u\|_{L^{\infty}(\Omega_T)} \leq 1$,
\begin{equation*} 
	\beta^2(x,t) \leq \frac{\delta}{2^{1+\alpha}} (|x|+\sqrt{|t|})^{\alpha} \quad \text{for all } (x,t) \in \overline{\Omega_T\cap Q_1},
\end{equation*}  
and
\begin{equation*}
	|g(x,t)| \leq \frac{\delta}{2}(|x| + \sqrt{|t|})^{2+\alpha} \quad \text{for all }  (x,t) \in \partial_s \Omega_T \cap Q_1.
\end{equation*}
Furthermore, by the scaling argument, we may assume that 
\begin{align*} 
	\beta^1(x,t) &\leq \frac{\delta}{K2^{1+\alpha}} (|x|+\sqrt{|t|})^{\alpha} \quad \text{for all } (x,t) \in \overline{\Omega_T\cap Q_1} , 
\end{align*}  
and
\begin{equation*}
	\|\partial \Omega\|_{C^{1,\alpha}(0)} \leq \frac{\delta}{2\sqrt{n}K},
\end{equation*}
where $K>0$ is the constant given in \Cref{lem2:gen_c2a}. Therefore, by \Cref{lem2:gen_c2a},  there exists a sequence $\{P_k\}_{k=-1}^{\infty}$ of the form $P_k(x)= \sum_{i=1}^n a_i^k x_i x_n$
such that for all $k \geq 0$, we have 
\begin{align*} 
	\|u-P_k\|_{L^{\infty}(\Omega_T \cap Q_{\eta^k})} \leq \eta^{k(2+\alpha)}, \quad  \sum_{i=1}^n |a^k_i-a^{k-1}_i| \leq  C_\star \eta^{(k-1)\alpha},  \quad \text{and} \quad F(D^2P_k,0,0)=0.
\end{align*}
According to the standard argument, we obtain a polynomial $P(x)=\sum_{i=1}^n a_i x_i x_n$ such that $P_k \to P$ satisfying
\begin{align*}
	\sum_{i=1}^n |a^k_i-a_i| \leq C_2\eta^{k \alpha } \quad \text{and} \quad  F(D^2 P,0,0)=0.
\end{align*}

Finally, for any $(x,t) \in \Omega_T \cap Q_1$, there exists $j \geq 0$ such that $\eta^{j+1} \leq \max\{ |x|, \sqrt{|t|} \}< \eta^j$. Then we conclude that
\begin{align*}
	|u(x,t)-P(x)| \leq |u(x,t)-P_j(x)|+|P_j(x) -P(x)|  \leq C(|x|+\sqrt{|t|})^{2+\alpha},
\end{align*}
which implies $u \in C^{2, \alpha}(0,0)$.
\end{proof}

\subsection{Proof of \Cref{thm:solv}}
The structure of the product $u^\gamma$ and $F$ in \eqref{prob:main} prevents the integration of boundary and interior regularity. The experience of converting the structure of the product $u^\gamma$ and $F$ into the H\"older continuity of fully nonlinear operator is also useful when proving global regularity.

Now we are ready to prove the main results of this paper.
\begin{proof}[Proof of \Cref{thm:solv}]
For each $(x_0,t_0) \in \overline{\Omega_T}$, we want to find a polynomial $P(x,t)$ with $\deg P \leq 2$ that satisfies $\|P\|_{C^2(\overline{\Omega_T}) } \leq C$ and
\begin{equation*}
	|u(x,t)-P(x,t)| \leq C(|x-x_0| + \sqrt{|t-t_0|})^{2+\alpha} \quad \mbox{for all } (x,t) \in \Omega_T.
\end{equation*}
If $(x_0,t_0) \in \partial_s \Omega_T$, then our goal $P$ is just itself from \Cref{thm:bdry_c2a}. Also, if $t_0 =-T$, then our goal $P$ is the polynomial derived from $C^{2,\tilde\alpha}$-regularity of $u_0$ on $\overline{\Omega_T}$.  So, it is enough to consider only the interior point $(x_0,t_0) \in \Omega_T$. Let $(x_1,t_1) \in \partial_p \Omega_T$ such that 
\begin{equation*}
	r\coloneqq \textnormal{dist}(\partial_p \Omega_T,(x_0,t_0)) = |x_0-x_1| + \sqrt{ |t_0-t_1|}.
\end{equation*}
If $t_1=-T$, the proof is the same as for the uniformity parabolic equations, so we only consider the case $t_1>-T$. By \Cref{thm:bdry_c2a}, there exist constants $\alpha_1 \in (0,\tilde\alpha)$, $A>0$, and a polynomial $P_1(x)$ with $\deg P_1 \leq 2$ and  $F(D^2P_1,x_1,t_1)=0$ such that 
\begin{equation*}
		|DP_1(x_1,t_1)| + \|D^2P_1\| \leq A
	\end{equation*}
and
\begin{equation} \label{u-p1}
	|u(x,t)-P_1(x)|\leq A (|x-x_1|+ \sqrt{|t-t_1|} )^{2+\alpha_1}
\end{equation}
for all $(x,t) \in \Omega_T \cap Q_1(x_1,t_1)$ which is extensible throughout $\Omega_T$. Since $u$ is a viscosity solution of \eqref{eq:convert}, the function $v=u-P_1$ is a viscosity solution of 
\begin{equation*}
	v_t = \overline{F}(D^2 v,x,t)   \quad \text{in } \Sigma,
\end{equation*}
where 
\begin{equation*}
	\overline{F} (M,x,t) \coloneqq  [\textnormal{dist}(x,\partial \Omega)]^\gamma \widetilde{F}(M +D^2P_1,x,t) 
\end{equation*}
and $\Sigma=\{(x,t):|x-x_0|+\sqrt{ |t-t_0|} < r/2 \}$. Since there exists $\theta \in \mathbb{R}$ such that $\overline{F}(\theta I_n, x_0,t_0) =0$ and the function $w(x,t) =v(x,t) -\theta |x|^2/2$ is viscosity solution of 
\begin{equation*}
	w_t = \overline{F}(D^2 w+ \theta I_n,x,t)  \quad \text{in } \Sigma, 
\end{equation*}
we may assume that $\overline{F}(O_n,x_0,t_0)=0$.

Since the operator $\overline{F}$ is uniformly parabolic in $\Sigma$, by the interior $C^{2,\alpha}$-estimates for uniformly parabolic equations, there exists a polynomial $P_2(x,t)$ with $\deg P_2 \leq 2$ such that 
\begin{equation} \label{norm_p2}
	r |DP_2(x_0,t_0)| + r^2\|D^2P_2\| +r^2 |\partial_t P_2 | \leq B \left( \|v\|_{L^{\infty}(\Sigma)} + r^2\| \overline{\beta}^2\|_{C^{\alpha}(\overline{\Sigma})} \right)
\end{equation}
and
\begin{equation} \label{v_int} 
	|v(x,t)-P_2(x,t)|  \leq B \bigg( \frac{ \|v\|_{L^{\infty}(\Sigma)}}{r^{2+\alpha}} + \frac{\| \overline{\beta}^2\|_{C^{\alpha}(\overline{\Sigma})} }{r^{\alpha}} \bigg) (|x-x_0| + \sqrt{|t-t_0|})^{2+\alpha} 
\end{equation}
for all $(x,t) \in \Sigma$, where $B>0$ is a constant depending only on $n$, $\lambda$, $\Lambda$, $\gamma$, $\alpha$, and $\|\beta^1\|_{C^{\alpha}(\overline{\Omega_T})}$ and $\overline{\beta}^2 \coloneqq  \beta^1 \|D^2P_1\| + \beta^2$.
We can assume that $\alpha<\alpha_1$. Then, \eqref{u-p1} implies that
\begin{align} 
	|v(x,t)| 
	& \leq AC r^{2+\alpha} \quad \text{for all }(x,t) \in \Sigma  \label{est_v}
\end{align}
and we also see
\begin{equation} \label{est_vt-Lv}
	|\overline{\beta}^2(x,t)| \leq C r^{\alpha} \quad \text{for all } (x,t) \in \overline{\Sigma}.
\end{equation}
Thus, combining \eqref{norm_p2}, \eqref{est_v}, and \eqref{est_vt-Lv} gives  
\begin{equation} \label{norm_pk2'}  	
	|P_2(x_0,t_0)|+ r |DP_2(x_0,t_0)| + r^2\|D^2P_2\| +r^2 |\partial_t P_2 | \leq ABC r^{2+\alpha}.
\end{equation}
Moreover, combining \eqref{v_int}, \eqref{est_v}, and \eqref{est_vt-Lv} gives  
\begin{align*}
	|u(x,t)-P_1(x)-P_2(x,t)| &\leq ABC (|x-x_0| + \sqrt{|t-t_0|})^{2+\alpha} \quad \text{for all } (x,t) \in \Sigma. 
\end{align*}
On the other hand, combining \eqref{u-p1} and \eqref{norm_pk2'} gives
\begin{align*}
	|u(x,t)-P_1(x)-P_2(x,t)| 
	&\leq  ABC(|x-x_0|+ \sqrt{|t-t_0|} )^{2+\alpha} \quad \text{for all $(x,t) \in \Omega_T \setminus \Sigma$.}
\end{align*}
Thus, $P \coloneqq P_1+P_2$ is the desired polynomial.
\end{proof}


%
%
\bibliographystyle{abbrv}

\end{document}